\documentclass[11pt]{amsart}
\usepackage{breqn,color}
\usepackage{amsmath,amsthm,amssymb,amsfonts,mathtools,bm,bbm,array,float,mathdots,dsfont,mathrsfs,latexsym}
\usepackage{caption,subcaption,multirow,longtable,lscape,wrapfig,tikz,tikz-cd,comment}
\usepackage[colorlinks=true, pdfstartview=FitV,linkcolor=blue,citecolor=blue,urlcolor=blue]{hyperref}
\usepackage{fullpage}
\usepackage{xcolor}
\usepackage[colorinlistoftodos]{todonotes}
\allowdisplaybreaks

\newtheorem {theorem}    {Theorem}[section]

\newtheorem {proposition}[theorem]    {Proposition}

\theoremstyle{definition}
\newtheorem{definition}[theorem]{Definition}

\newtheorem{ex}[theorem]{Example}

\def\x{\mathbf{x}}
\def\y{\mathbf{y}}

\newenvironment{red}{\relax\color{red}}{\hspace*{.5ex}\relax}
\newenvironment{blue}{\relax\color{blue}}{\hspace*{.5ex}\relax}
\newcommand{\ber}{\begin{red}}
\newcommand{\er}{\end{red}}
\newcommand{\beb}{\begin{blue}}
\newcommand{\eb}{\end{blue}}
\newcommand{\seteq}{\mathbin{:=}}

\numberwithin{equation}{section}
\numberwithin{figure}{section}
\numberwithin{table}{section}
\allowdisplaybreaks
\begin{document}

\title{Machine Learning Mutation-Acyclicity of Quivers}

\date{\today}

\author[K.~T.~K. Armstrong-Williams]{Kymani T.~K. Armstrong-Williams}
\address{Centre for Theoretical Physics, Queen Mary University of London, E1 4NS, U.K.}
\email{k.t.k.armstrong-williams@qmul.ac.uk}

\author[E. Hirst]{Edward Hirst}
\address{Centre for Theoretical Physics, Queen Mary University of London, E1 4NS, U.K.}
\email{e.hirst@qmul.ac.uk}

\author[B. Jackson]{Blake Jackson}
\address{Department of
Mathematics, University of Connecticut, Storrs, CT 06269, U.S.A.}
\email{blake.jackson@uconn.edu}

\author[K.-H. Lee]{Kyu-Hwan Lee}

\address{Department of
Mathematics, University of Connecticut, Storrs, CT 06269, U.S.A. \hfill \break \indent Korea Institute for Advanced Study, Seoul 02455, Republic of Korea}
\email{khlee@math.uconn.edu}


\begin{abstract}
Machine learning (ML) has emerged as a powerful tool in mathematical research in recent years. 
This paper applies ML techniques to the study of quivers---a type of directed multigraph with significant relevance in algebra, combinatorics, computer science, and mathematical physics. 
Specifically, we focus on the challenging problem of determining the mutation-acyclicity of a quiver on 4 vertices, a property that is pivotal since mutation-acyclicity is often a necessary condition for theorems involving path algebras and cluster algebras. 
Although this classification is known for quivers with at most 3 vertices, little is known about quivers on more than 3 vertices.
We give a computer-assisted proof of a theorem to prove that mutation-acyclicity is decidable for quivers on 4 vertices with edge weight at most 2.
By leveraging neural networks (NNs) and support vector machines (SVMs), we then accurately classify more general 4-vertex quivers as mutation-acyclic or non-mutation-acyclic.
Our results demonstrate that ML models can efficiently detect mutation-acyclicity, providing a promising computational approach to this combinatorial problem, from which the trained SVM equation provides a starting point to guide future theoretical development. \\
Report Number: QMUL-PH-24-27 
\end{abstract}


\maketitle


\vspace*{-0.5 cm}

\section{Introduction}

In the current era of artificial intelligence (AI) transforming nearly every aspect of society and science, mathematics research is no exception.
The efficiency and capability of machine learning (ML) can undoubtedly be utilized in mathematics research. 
Indeed, intensive research has been conducted on ``ML mathematics'' to explore whether ML can detect underlying mathematical structures.
Through applying ML tools to various datasets generated from mathematical objects, high accuracies (often exceeding 98\%) have been achieved in classifying mathematical objects according to their invariants and properties. 
Examples include algebraic curves \cite{ Hashemi:2024azx, He_2022, He:2022pqn, He_2023}, special holonomy manifolds \cite{Aggarwal:2023swe, Alawadhi:2023gxa, Ashmore:2023ajy, Berglund:2023ztk, Berman:2021mcw, Hirst:2023kdl}, branes \cite{Arias-Tamargo:2022qgb, Bao:2021olg, Capuozzo:2024vdw, Chen:2022jwd, Seong:2023njx}, and number fields \cite{Amir:2022sab,He_2022_n}.
Interpreting what the machine learns can sometimes lead to new discoveries or perspectives in the study of mathematical objects. 
The recent discovery of {\em murmuration} phenomena in arithmetic is a prime example \cite{quanta, He:2022pqn}. 
This proposes a new paradigm in mathematics research: generate datasets, apply ML tools, interpret the results, gain new perspectives, formulate conjectures, and prove them.

In this paper, we apply this framework to the problem of quiver mutations. 
A {\em quiver} is a directed multigraph without loops or directed 2-cycles, and quivers are ubiquitous in mathematics, mathematical physics, and computer science, as they efficiently encode fundamental skew-symmetric structures. 
The {\em mutation} of a quiver is a combinatorial operation that originates from cluster algebras \cite{Fomin2002clusterI}, where it is an essential component, and from quantum field theory, where it relates to Seiberg duality \cite{Feng_2001, Seiberg_1995}. 
The iterative application of these mutations generates a rich and intricate algebraic combinatorial structure.

When a quiver is obtained from another via successive mutations, we say the two quivers are {\em mutation-equivalent}. 
In the theory of cluster algebras and related topics, it is crucial to know whether a quiver is mutation-equivalent to a quiver with no oriented cycles, called an {\em acyclic} quiver. 
When a quiver is mutation-equivalent to an acyclic quiver, it is ``in control,'' and we understand relatively well the constructions based on such a quiver from various tools coming from representation theory and combinatorics.
Otherwise, it is ``out of control,'' and our understanding of the related structures is limited. 
Thus, knowing whether a quiver is mutation-equivalent to an acyclic quiver is often a question with a highly consequential answer.

However, even for the problem of determining {\em mutation-acyclicity}---let alone general mutation equivalence---we currently lack practical methods or algorithms when the number of vertices is greater than 3. 
With the advent of new approaches from the ML framework, we investigate in this paper whether ML can assist in addressing this problem. 
After generating datasets from mutations of quivers, we apply ML tools to this problem from both supervised and unsupervised learning. 
More precisely, we apply neural networks (NNs), support vector machines (SVMs), and principal component analysis (PCA). 
This extends work from \cite{Bao:2020nbi, Cheung:2022itk, Dechant:2022ccf} where NNs were first applied to differentiate mutation classes and their representations, now in this work with a focus on the mutation-acyclicity property.

In the construction of quiver data for analysis, we prove a new Theorem, Theorem \ref{ca_theorem}, via computer-assisted means; classifying all rank 4 quivers with edge weight at most 2 according to the mutation-acyclicity property.
The results of our following ML experiments clearly show that ML architectures can be used to detect mutation-acyclicity efficiently. 
We achieve high accuracy in distinguishing mutation-acyclic quivers from non-mutation-acyclic ones using both NNs and SVMs. 
In fact, it was these first promising ML results that motivated the authors to perform the work required to prove the new Theorem.
The SVMs provide directly interpretable equations (viewed as separating hypersurfaces) much like the Markov constant for rank 3 quivers \cite{BBH2011Cluster}. 
These trained equations should provide theoretical insight into the underlying structure of high-rank mutation-acyclicity invariants, where we leave full analysis and design of the invariant to future work. 

For the ML experiments of this paper, quiver data was generated using the \texttt{SageMath} \cite{sage} package \cite{musiker2011compendium}, and graph analysis was conducted using \texttt{networkx} \cite{Hagberg2008ExploringNS} and \texttt{graph-tool} \cite{peixoto_graph-tool_2014}. 
The subsequent machine learning tasks were performed in \texttt{python}, utilizing the \texttt{scikit-learn} \cite{sklearn} and \texttt{tensorflow} \cite{tensorflow2015whitepaper} packages. 
Scripts were adapted from previous works \cite{Cheung:2022itk, Dechant:2022ccf}, with high-performance computing resources \cite{apocrita}. 
All scripts and data are made available on this work’s respective \href{https://github.com/KTKAW/MACHINE_LEARNING_MUTATION_ACYCLICITY_OF_QUIVERS.git}{\texttt{GitHub}}\footnote{\href{https://github.com/KTKAW/MACHINE_LEARNING_MUTATION_ACYCLICITY_OF_QUIVERS.git}{\texttt{https://github.com/KTKAW/MACHINE_LEARNING_MUTATION_ACYCLICITY_OF_QUIVERS.git}}}.

After this introduction, §\ref{sec:bg} covers background material. 
In §\ref{subsec:ca}, quiver mutation and cluster algebras are defined, and some results on mutation-acyclicity are presented, including our new result (Theorem \ref{ca_theorem}). 
In §\ref{subsec:ml}, the ML methods and tools that are adopted in this paper are reviewed, including NNs, SVMs, and PCA. 
§\ref{sec:data} explains how we generate our datasets from quiver mutations. 
The next section, §\ref{sec:ml}, contains the results of our ML investigations. 
§\ref{sec:nn} summarises what we obtain from NNs. 
§\ref{sec:sv} highlights separating hyperplanes coming from SVMs. 
In the final section, §\ref{sec:conc}, we make concluding remarks. 

\section{Background} \label{sec:bg}
Cluster algebras provide a well-motivated arena for testing the efficacy of machine learning (ML) methods, chiefly due to the combinatorial nature of the quivers and their matrix representations. 
A prominent problem within cluster algebras is identifying whether a given quiver is equivalent to an acyclic\footnote{In graph theoretic terms, we are interested in \textit{strongly} acyclic weighted directed graphs.} quiver or not under the mutation operation. 
Whilst this is infeasible by eye, this paper examines how effective ML architectures are at learning this property.
In this background section, the relevant mathematical details underpinning cluster algebras and this mutation-acyclicity problem are introduced, along with the implemented ML methods.
 
\subsection{Cluster Algebras} \label{subsec:ca} \mbox{}\\
Cluster algebras are subrings of the field of rational functions in $n$ commuting variables over $\mathbb{Q}$.
They often appear as coordinate rings of certain geometric objects, and were formalized by Fomin and Zelevinsky \cite{Fomin2002clusterI} in order to study their (dual) canonical bases and total positivity.
Since their introduction, cluster algebras have found applications in various areas of mathematics and mathematical physics, including representation theory, combinatorics, algebraic geometry, dynamical systems, knot theory, and string theory. 

What makes cluster algebras distinctive is that they are defined recursively via a local process, called \textit{mutation}, with some initial algebraic and combinatorial data, called an \textit{initial seed}.
By repeatedly mutating the initial seed (sometimes indefinitely), we can generate the full list of generators required to define a cluster algebra.

\begin{definition} \label{def-quiver}
    A \textit{quiver} is a finite, directed multigraph $Q = (Q_0, Q_1)$. We always assume that a quiver does not have loops or oriented 2-cycles. 
    The elements of $Q_0 = \{1,\ldots, n\}$ are the \textit{vertices} of $Q$ and the elements of $Q_1$ are the \textit{arrows} of $Q$.
    In this case, we call $n$ the \textit{rank} of $Q$.
    There exist two maps, the \textit{source map} and the \textit{target map} $s, t: Q_1 \to Q_0$, where $s(\alpha)$ returns the vertex at the source of $\alpha$ and $t(\alpha)$ returns the vertex at the target of the arrow $\alpha$.
    We say a quiver is \textit{acyclic} if it contains no directed cycles of any length.
\end{definition}

\begin{definition} \label{def-quiver-mutation}
    Given a quiver $Q$ and a vertex $k$, we can define a new quiver $\mu_k(Q)$, called the \textit{mutation of $Q$ at vertex $k$}, in the following way:
    \begin{enumerate}
        \item for each directed 2-path $i \to k \to j$ in $Q$, add an arrow $i \to j$ in $\mu_k(Q)$,
        \item reverse the direction of all arrows incident to $k$,
        \item pairwise delete any oriented 2-cycles in $\mu_k(Q)$ which have appeared as a result of step 1.
    \end{enumerate}
    We say a quiver $Q'$ is \textit{mutation-equivalent} to $Q$ if there is a sequence of vertices $[k_1,k_2\ldots,k_\ell]$ such that $Q' = \mu_{k_\ell}( \cdots \mu_{k_2}(\mu_{k_1}(Q)) \cdots )$: in other words, we can pass from $Q$ to $Q'$ by successive mutations at the vertices $k_1, k_2, ..., k_\ell$.
    The \textit{mutation class} $[Q]$ of $Q$ is the collection of all quivers $Q'$ that are mutation-equivalent to $Q$.
\end{definition} 

\begin{ex}
The following is an example of mutation in a rank 4 quiver: \[Q = \begin{tikzcd}
    & 1 \\
    & 2 \\
    3 & & 4 
    \arrow[from=1-2, to=2-2, "2"]
    \arrow[from=2-2, to=3-1]
    \arrow[from=2-2, to=3-3]
    \arrow[from=3-1, to=1-2]
    \arrow[from=3-1, to=3-3]
    \end{tikzcd}
    \quad \longrightarrow \qquad
    \mu_2(Q) = \begin{tikzcd}
    & 1 \\
    & 2 \\
    3 & & 4 
    \arrow[from=2-2, to=1-2, "2"]
    \arrow[from=3-1, to=2-2]
    \arrow[from=3-3, to=2-2]
    \arrow[from=1-2, to=3-1]
    \arrow[from=3-1, to=3-3]
    \arrow[from=1-2, to=3-3, "2"]
    \end{tikzcd}\]
Arrows with the label $2$ indicate two arrows between the vertices.
As seen from the example and definition, mutating a quiver $Q$ at any vertex $k$ produces another quiver. 
Moreover, mutation is an involution, meaning that $\mu_k(\mu_k(Q)) = Q$ for any vertex $k$.
It should also be noted that quivers are in bijective correspondence with skew-symmetric integral matrices. 
If $Q$ is a rank $n$ quiver with $Q_0 = \{ 1, \ldots, n\}$, then $Q$ corresponds to a $n \times n$ skew-symmetric matrix $B = B(Q) = (b_{ij})$ where \[ b_{ij} = (\# \text{ arrows from } i \text{ to } j) - (\# \text{ arrows from } j \text{ to } i). \]
This is the signed adjacency matrix of the quiver when viewed as a directed graph.
There is also a notion of mutation of skew-symmetric matrices, and this definition agrees with the definition of mutation of quivers.
Since the correspondence between quivers and skew-symmetric integral matrices is bijective and mutation works the same in both cases, we prefer to use the quiver perspective when visualizing quivers and the skew-symmetric matrix perspective for ML. 

For $Q$ at the beginning of this example, we have \[ B(Q) = \begin{pmatrix}
    0 & 2 & -1 & 0 \\
    -2 & 0 & 1 & 1 \\
    1 & -1 & 0 & 1 \\
    0 & -1 & -1 & 0
\end{pmatrix} \qquad \text {and} \qquad \mu_2(B(Q)) = B(\mu_2(Q)) = \begin{pmatrix}
    0 & -2 & 1 & 2 \\
    2 & 0 & -1 & -1 \\
    -1 & 1 & 0 & 1 \\
    -2 & 1 & -1 & 0
\end{pmatrix}\]
\end{ex}

\begin{definition} \label{def-cluster-algebra}
    Let $\x = \{x_1,\ldots,x_n\}$ be a set of $n$ variables and $Q$ be a quiver on $n$ vertices, where the vertex labeled $k$ is understood to correspond with the variable $x_k$. 
    Let $\mathcal{F} = \mathbb{Q}(x_1,\ldots,x_n)$ be the field of rational functions in $\x$.
    \textit{Mutation of the seed $(\x,Q)$ at vertex $k$} is the pair $\mu_k(\x,Q) = (\x', Q')$ where $Q'=\mu_k(Q)$ and $\x'  \seteq  \mu_k(\x)=(\x \setminus \{ x_k \} ) \cup \{x_k'\}$ and $x_k' \in \mathcal{F}$ is defined by the \textit{exchange relation}: 
    \[ x_kx_k' = \prod_{ \alpha  :  s(\alpha) = k  }x_{t(\alpha)} + \prod_{ \alpha  :  t(\alpha) = k  }x_{s(\alpha)}. \]
    We call $\x'' \in \mathcal F$ a \textit{cluster} if $\x'' =  \mu_{k_\ell} \cdots \mu_{k_2}\mu_{k_1}(\x)$ for some sequence $[k_1,k_2\ldots,k_\ell]$ and elements $x_k'' \in \x''$ \textit{cluster variables}.   The \textit{cluster algebra (of rank $n$)} $\mathcal{A} = \mathcal{A}(\x,Q)$ is the $\mathbb{Z}$-subalgebra of $\mathcal{F}$ generated by all the cluster variables.
\end{definition}

The role that the initial combinatorial data $(\x, Q)$ plays in the structure of $\mathcal{A}(\x,Q)$ cannot be overstated.
The addition or removal of a single arrow of $Q$ can determine whether the resulting algebra consists of finitely many or infinitely many cluster variables, and many theorems in the cluster algebra literature depend on the mutation class of $Q$.
Specifically, there is a collection of results that require the initial quiver $Q$ to be acyclic, so that the cluster algebra is well-behaved.

\subsubsection{Quiver Mutation-Acyclicity}\mbox{}\\
In this section, we define mutation-acyclicity and motivate it with examples from the literature. 
Going through each of these examples in detail is beyond the scope of this paper, and we direct readers to the references in this section should they want to investigate further.

\begin{definition} \label{def-mutation-acyclic}
    Let $Q$ be a quiver. We say that $Q$ is \textit{mutation-acyclic} if there exists at least one acyclic quiver in $[Q]$, or equivalently, if $Q$ is mutation-equivalent to an acyclic quiver $Q'$.
    If there are no acyclic quivers in $[Q]$, then we say $Q$ is \textit{non-mutation-acyclic}.
\end{definition}

Researchers care about acyclic and mutation-acyclic quivers for a variety of reasons. 
For one, path algebras over acyclic quivers have finite vector space dimension \cite{Assem2006Elements,Schiffler2014Quiver}.
These are important objects of study since every hereditary Artin algebra is Morita equivalent to a path algebra over some quiver, and finite matrices can represent the maps between finite-dimensional algebras.
Buan, Marsh, and Reiten show that, for a cluster algebra arising from an acyclic quiver, cluster mutations correspond to tilting objects in the cluster category over the path algebra of $Q$ \cite{BMR2008Cluster}.
In terms of the cluster algebras this means that, if $Q$ is acyclic, the denominator vectors of the non-initial cluster variables of $\mathcal{A}(\x,Q)$ correspond to the dimension vectors of the indecomposable rigid representations of $Q$ over the path algebra $kQ$, where $k$ is an algebraically closed field; these dimension vectors are the so-called \textit{real Schur roots}.
Furthermore, these denominator vectors (real Schur roots) correspond to the positive roots of the root system of the Kac--Moody algebra associated with $Q$ \cite{Bongartz1983Algebras,Kac1980Infinite}.
It was also shown that the real Schur roots coincide with the $c$-vectors of the cluster algebra $\mathcal{A}(\x,Q)$ when $Q$ is acyclic \cite{Chavez2015Acyclic}.
There has been research attempting to understand the $c$-vectors of cluster algebras generated from acyclic quivers for this reason, including some work by the last two authors \cite{EJLN2024Geometry,Felikson2018Acyclic,LL2021Correspondence,LLM2023Geometric}.

Another concept is that of \textit{reddening sequences} in cluster algebras with coefficients, and it is known that mutation-acyclic quivers admit a reddening sequence.
In his paper introducing reddening sequences (maximal green sequences), Keller \cite{Keller2017Quiver} shows that reddening sequences produce explicit formulas for the refined Donaldson--Thomas invariants introduced by Kontsevich and Soibelman \cite{Kontsevich2008Stability}.
The existence of reddening sequences is also a sufficient condition for the existence of a theta basis for the upper cluster algebra (a sought-after canonical basis) \cite{GHKK2018Canonical}.

Other research focuses on the class of \textit{locally acyclic cluster algebras} \cite{Muller2013Locally} where the cluster algebra can be ``covered'' by cluster algebras generated by acyclic quivers.
These have many desirable properties, chief among them that the cluster algebra is equal to the upper cluster algebra ($\mathcal{A} = \mathcal{U}$), which is the intersection of Laurent rings of the cluster variables and has a geometric interpretation \cite{Muller2014AU}.

Given the desirable properties that cluster algebras generated from acyclic quivers possess, one might want to know when a given quiver is mutation-equivalent to an acyclic quiver; in other words, when a quiver is mutation-acyclic.
This is because $\mathcal{A}(\x,Q') \cong \mathcal{A}(\x,Q)$ when $Q$ and $Q'$ are mutation-equivalent.
However, deciding whether or not a given quiver is mutation-acyclic is a difficult problem; in fact, there is no known algorithm or invariant for detecting mutation-acyclicity in quivers with more than 3 vertices (except in special cases).
We now survey the known results that can be used to determine mutation-acyclicity.

\subsubsection{Quiver Mutation-Acyclicity in Rank 1, 2, and 3}\mbox{}\\
The problem is trivial for rank 1 and rank 2 quivers since they are all, by definition, acyclic, and so every quiver of rank 1 or 2 is mutation-acyclic.
The first non-trivial case occurs when $Q$ is a cyclic quiver of rank 3.

In 2008, Assem, Blais, Br\"ustle, and Samson \cite{assem2008mutation} described an algorithm for detecting mutation-acyclicity\footnote{This algorithm, along with the fact that the acyclic quivers in a mutation-class form a connected component of the exchange graph, can actually be used to detect mutation-equivalence in rank 3 quivers. This algorithm results in a minimal representative of a quiver's mutation class.} in rank 3 quivers.
Let \[Q = \begin{tikzcd}[row sep={4em,between origins}, column sep={4em,between origins}]
    & 2 &   \\
    1 & & 3 
    \arrow[from=2-1, to=1-2, "x"]
    \arrow[from=1-2, to=2-3, "y"]
    \arrow[from=2-3, to=2-1, "z"]
    \end{tikzcd}\]
be a cyclic quiver on 3 vertices. 
The algorithm described is simple: mutate at the vertex opposite the largest collection of edges.
So, if $z$ were the largest collection of edges in $Q$, mutate at vertex $2$.
If there is more than one largest collection of edges, the one you pick to mutate across from doesn't matter.
This process continues until either an acyclic quiver is reached or mutation at any vertex does not decrease the total number of edges ($x+y+z$).

In 2011, Beineke, Br\"ustle, and Hille \cite{BBH2011Cluster} were able to circumvent the algorithm described above and produced a mutation-invariant constant (called the Markov constant) which is used to classify rank 3 quivers as mutation-acyclic or not.
Let \[C(x,y,z) = x^2 + y^2 + z^2 - xyz\] be the \textit{Markov constant} associated to a rank 3 cyclic quiver $Q$ (as above). 
A partial statement\footnote{The full pair of theorems has a few more conditions, but for our purposes, the statement below is sufficient.} of their two main theorems is below. 
\begin{theorem}\cite[Theorem 1 and 2]{BBH2011Cluster}\label{theorem:subquiver}
    Let $Q$ be a cyclic quiver on 3 vertices with numbers of arrows given by $x,y,z \in \mathbb{Z}_{\geq 0}$. 
    Then $Q$ is mutation-acyclic if any of the following are satisfied:
    \begin{enumerate}
        \item $C(x,y,z) > 4$ or $\min\{x,y,z\} < 2.$
    \end{enumerate}
The quiver $Q$ is non-mutation-acyclic if any of the following are satisfied:
    \begin{enumerate}
        \item $C(x,y,z) < 0$
        \item $C(x,y,z) \leq 4$ and $x,y,z \geq 2$.
    \end{enumerate}
\end{theorem}

A consequence of this theorem is that the only non-mutation-acyclic quiver of rank 3 with at most 2 arrows between any pair of vertices is the Markov quiver:
\[M = \begin{tikzcd}[row sep={4em,between origins}, column sep={4em,between origins}]
    & 2 &   \\
    1 & & 3 
    \arrow[from=2-1, to=1-2, "2"]
    \arrow[from=1-2, to=2-3, "2"]
    \arrow[from=2-3, to=2-1, "2"]
    \end{tikzcd}\]
The ability to check whether a rank 3 quiver is non-mutation-acyclic using the Markov constant plays a large part in the proof of Theorem~\ref{ca_theorem} and in generating data for our ML experiments.

As a foreshadowing of the ML techniques used later in this paper, Theorem~\ref{theorem:subquiver} has a geometric interpretation.
There are two ``critical values'' for the Markov constant: 0 and 4.
If $C(x,y,z) > 4$ (or there is a single/missing arrow), then the quiver is mutation-acyclic; and if $C(x,y,z) < 0$, then the quiver is non-mutation-acyclic.
By plotting the triple $(x,y,z)$ of edge weights in the first octant, we can see that $C(x,y,z) = 0, 4$ defines two separating (or classifying) nonlinear hypersurfaces in $\mathbb{R}^3$ (Figure~\ref{fig-hypersurfaces}).
The singularity in Figure~\ref{fig-hypersurfaces4} occurs at the point $(2,2,2)$, which is the Markov quiver.
This geometric classification was one of the primary motivations for the Support Vector Machine analysis in §\ref{sec:sv}.

\begin{figure}[!h]
    \centering
    \begin{subfigure}[t]{0.33\textwidth}
        \centering
        \includegraphics[width = \textwidth]{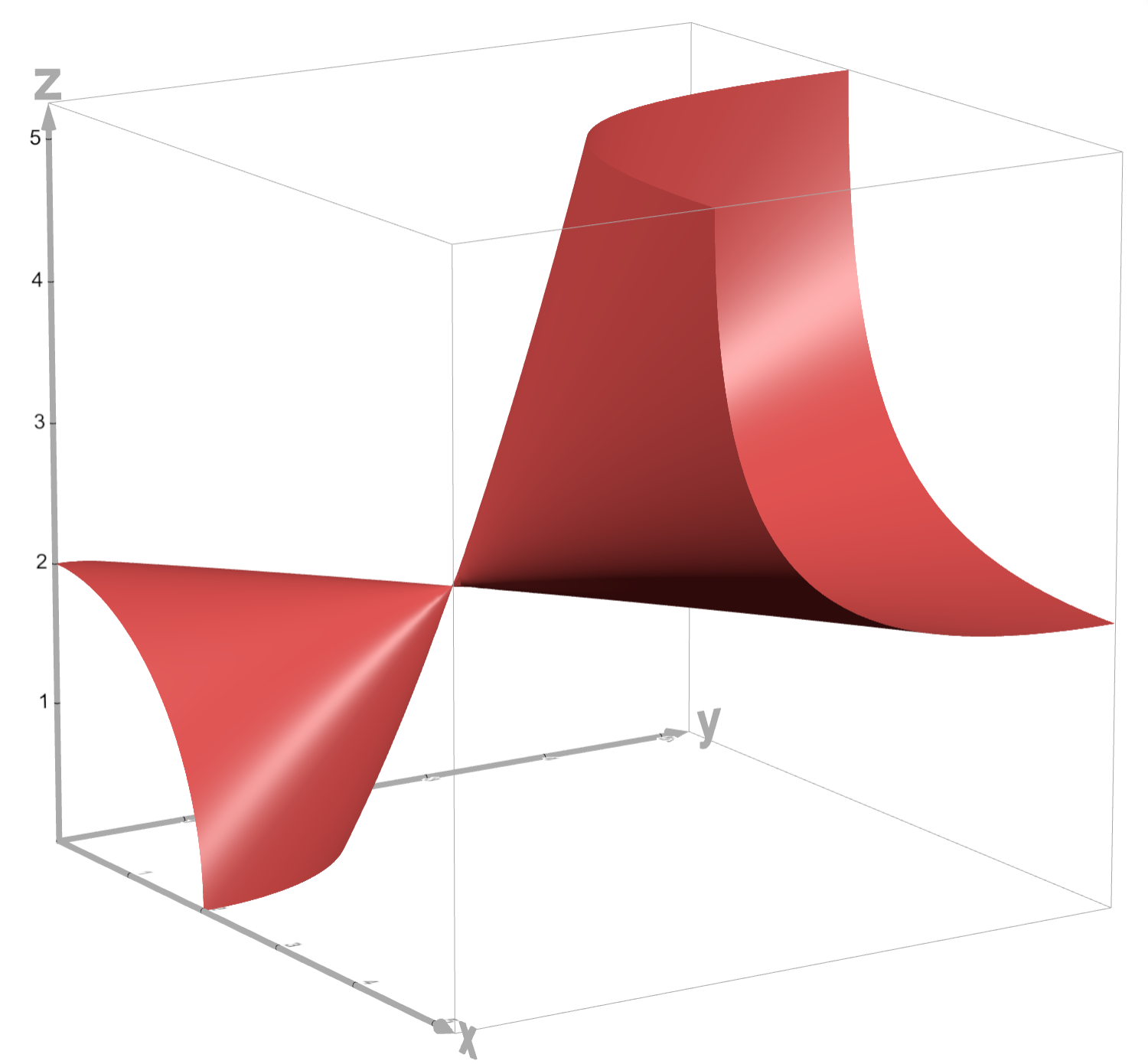}
        \caption{The hypersurface $C(x,y,z) = 4$}\label{fig-hypersurfaces4}
    \end{subfigure}%
    ~ 
    \begin{subfigure}[t]{0.33\textwidth}
        \centering
        \includegraphics[width = \textwidth]{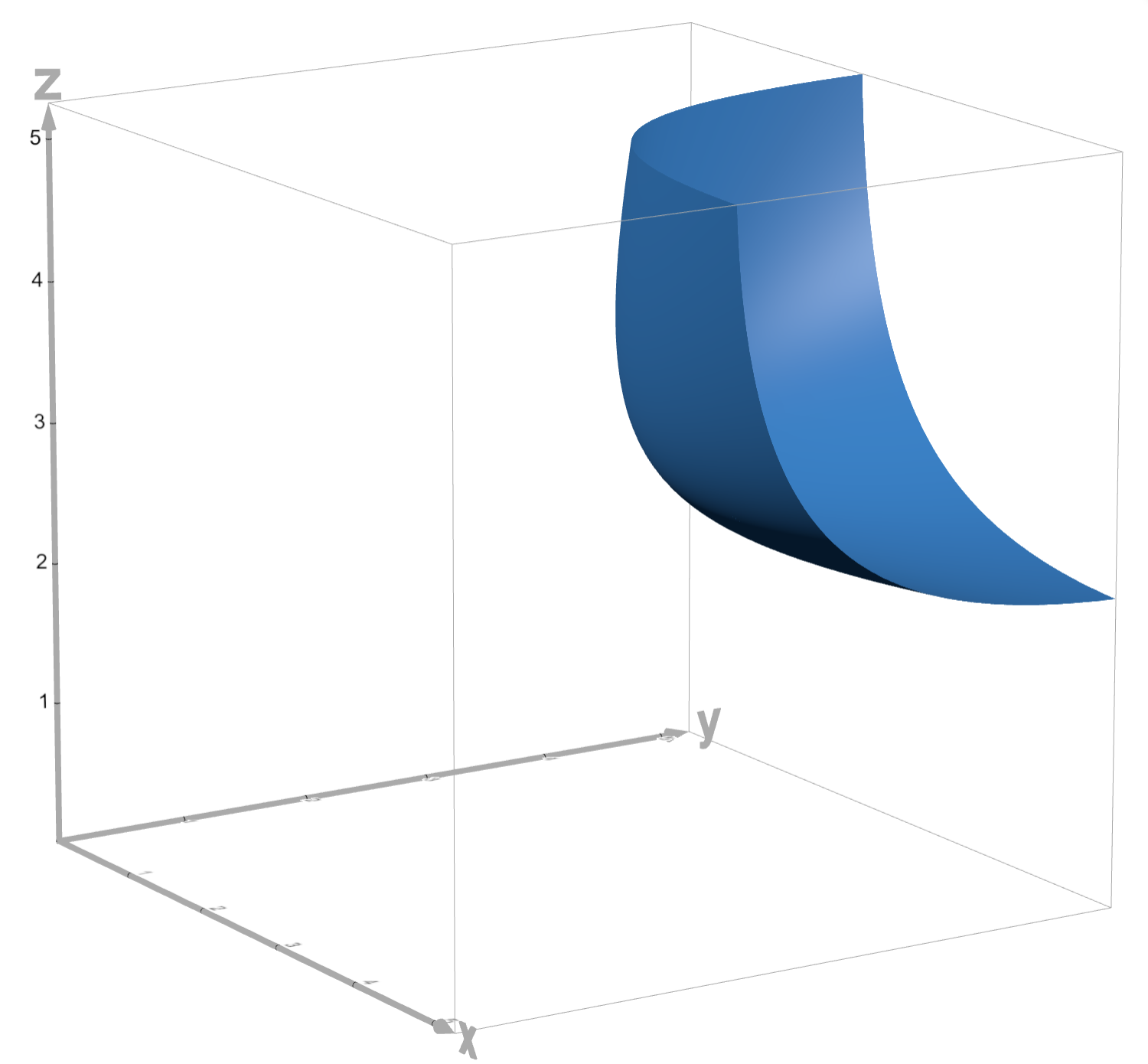}
        \caption{The hypersurface $C(x,y,z) = 0$}
    \end{subfigure}
    ~ 
    \begin{subfigure}[t]{0.33\textwidth}
        \centering
        \includegraphics[width = \textwidth]{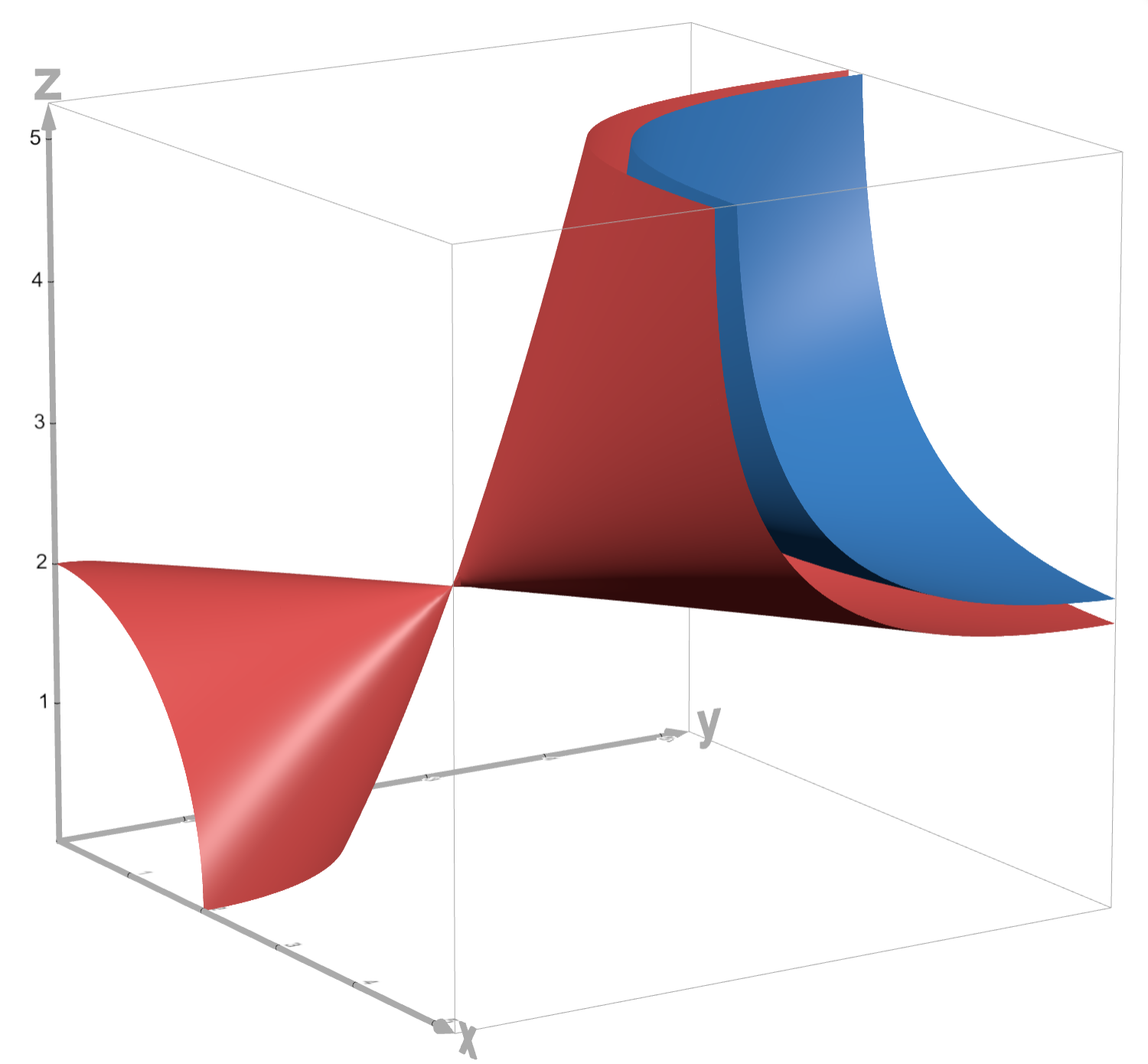}
        \caption{The hypersurfaces together}
    \end{subfigure}
    \caption{The Markov constant classification visualized as separating hypersurfaces}\label{fig-hypersurfaces}
\end{figure} 

\subsubsection{Known Results on Quiver Mutation-Acyclicity in Rank 4}\mbox{}\\
In rank 4, few existing results can help researchers determine if an arbitrary quiver $Q$ is mutation-acyclic.
There is currently no algorithm or mutation-invariant that can be used to decide mutation-acyclicity in ranks 4 and above, and there is reason to believe that no such algorithm exists for arbitrary quivers \cite{fomin_long_2023}.
However, as we will see later, we find evidence of an underlying (possibly mutation-invariant) combinatorial structure that the machine can learn to distinguish mutation-acyclic and non-mutation-acyclic rank 4 quivers.

In what follows, we assume that $Q$ only has one connected component; if $Q$ is a rank 4 quiver that has more than one connected component, then the problem of mutation-acyclicity can be dealt with using the results in the previous section. 
The most applicable result in rank 4 (as well as higher ranks) is due to Buan, Marsh, and Reiten: Proposition~\ref{prop-BMR}.

\begin{definition}\label{def-full-subquiver}
    Let $Q$ be a quiver on $n$ vertices and $I \subseteq \{1,\ldots,n\}$ be a subset of the vertex set. 
    A quiver $Q_I$ is a \textit{full subquiver} of $Q$ if its vertex set is $(Q_I)_0 = I$ and its arrows are given by $ (Q_I)_1 = \{ \alpha \in Q_1 \ | \ s(\alpha),t(\alpha) \in I \}$.
\end{definition}

\begin{proposition} \cite[Corollary 5.3]{BMR2008Cluster} \label{prop-BMR}
    Let $Q$ be any quiver. 
    If any full subquiver of $Q$ is non-mutation-acyclic, then $Q$ is non-mutation-acyclic.
\end{proposition}
In other words, if $Q$ is a rank 4 quiver and contains a non-mutation-acyclic rank 3 quiver (a problem that has been dealt with in the previous section), then $Q$ itself is non-mutation-acyclic.
This result is often useful in practice, especially when dealing with quivers of higher rank, since a rank $n$ quiver contains ${n \choose 3}$ rank 3 subquivers, and Markov constants are easy to calculate.

On the other hand, there are rank 4 quivers in which all (proper) full subquivers are mutation-acyclic, but the quiver itself is non-mutation-acyclic. 
In fact, in his dissertation, Warkentin \cite{Warkentin2014Exchange} was able to provide an infinite family of examples of rank 4 quivers in which all of the proper subquivers are mutation-acyclic, but $Q$ is non-mutation-acyclic.
Specifically, he shows that quivers of the form \[Q = \begin{tikzcd}[row sep={4em,between origins}, column sep={4em,between origins}]
    1 & 2   \\
    4 & 3 
    \arrow[from=1-1, to=1-2, "a"]
    \arrow[from=1-2, to=2-2, "b"]
    \arrow[from=2-2, to=2-1, "a"]
    \arrow[from=2-1, to=1-1, "b"]
    \end{tikzcd}\]
where $a,b \geq 2$ are non-mutation-acyclic. 
While the ``box quivers'' contain no rank 3 non-mutation-acyclic subquivers, they are sometimes mutation-equivalent to quivers $Q'$ for which Proposition~\ref{prop-BMR} can be applied.

There is also a well-known pathological quiver called the ``dreaded torus'' that is obtained from the triangulation of a torus with one boundary component and one marked point on the boundary: \[Q = \begin{tikzcd}[row sep={4em,between origins}, column sep={4em,between origins}]
    & 4 \\
    & 2 \\
    1 & & 3 
    \arrow[from=3-1, to=3-3, shorten=0.2em]
    \arrow[from=3-3, to=2-2, shorten=0.2em]
    \arrow[from=3-1, to=2-2, shorten=0.2em]
    \arrow[from=1-2, to=3-1, shorten=0.2em]
    \arrow[from=1-2, to=3-3, shorten=0.2em]
    \arrow[from=2-2, to=1-2, "2"right, shorten=0.2em]
    \end{tikzcd}\]
The dreaded torus has the interesting property that mutation at any vertex produces an isomorphic copy of the dreaded torus quiver, making it effectively the only quiver in its mutation class (much like the Markov quiver).
Since it contains an oriented cycle, it is non-mutation-acyclic.
The existence of the dreaded torus (and possibly other pathological quivers) complicates the problem of deciding mutation-acyclicity using machines since Proposition~\ref{prop-BMR} is not applicable.
 
The goal of this paper is to utilize computational techniques to assist researchers in classifying rank 4 quivers into mutation-acyclic and non-mutation-acyclic.
To make theoretical progress on this problem, we show that the question is decidable for rank 4 quivers with at most 2 arrows between any pair of vertices.

\begin{theorem}\label{ca_theorem}
    Suppose we have a quiver $Q$ with 4 vertices, where the arrow weights take on values of 0, 1, or 2.
    Then $Q$ is non-mutation-acyclic if and only if one of the following conditions holds.
    \begin{enumerate}
        \item $Q$ contains a quiver in its mutation class that has a rank 3 non-mutation-acyclic subquiver. 
        \item $Q$ is isomorphic to the dreaded torus quiver.
    \end{enumerate} 
\end{theorem}
\begin{proof}
    The full details of this computer-assisted proof are available on \href{https://github.com/KTKAW/MACHINE_LEARNING_MUTATION_ACYCLICITY_OF_QUIVERS.git}{\texttt{GitHub}}; a summary of the logic is provided here.
    All quiver adjacency matrices for rank 4 quivers with arrow weights 0, 1, or 2 can be exhaustively generated, where the six entries of the upper triangle can take values from $\{-2,-1,0,1,2\}$, hence producing $5^6 = 15625$ matrices.
    Of these, those that were not (weakly) connected, and hence not truly rank 4, were removed, leaving $15104$ quivers.
    These adjacency matrices are still not unique up to graph isomorphism; therefore, a canonical form for the matrices was defined\footnote{Using lexicographical ordering of the matrix elements.} under this symmetry, and the set was reduced to $667$ isomorphism classes of quiver adjacency matrices.
    
    The representative quivers which are (strongly) acyclic are then trivially mutation-acyclic, 401 of the 667 were shown to be acyclic and were partitioned off into the mutation-acyclic side of the classification, leaving 266 unknown cases.
    Of these, 42 were shown to contain the Markov quiver as a subquiver, and two further cases were shown to be the 2-2-2-2 box quiver and dreaded torus.
    Partitioning these off resulted in 44 quivers on the non-mutation-acyclic side of the classification, with 222 remaining unknown cases.
    These 44 non-mutation-acyclic quivers were then each mutated exhaustively up to some set depth (8 in the case of this proof). 
    Of the quivers produced in this way, 2 matched one of the 222 remaining unknown cases, which were then labeled as non-mutation-acyclic, leaving 220 unknown cases.

    Each of these 220 unknown cases were then exhaustively mutated (using a breadth-first heuristic), producing their exchange graphs at ever greater depths until either an acyclic quiver was produced (and hence the initial quiver could be classified as mutation-acyclic) or a quiver which could be shown to be non-mutation-acyclic was produced (and hence the initial quiver could be classified as non-mutation-acyclic). 
    Additionally, a flag was added that if the entire exchange graph was generated without producing an acyclic quiver, the initial quiver was marked as non-mutation-acyclic.
    The checks for non-mutation-acyclicity were carried out by searching each produced quiver's 3-vertex subquivers, and, when the subquiver was a directed cycle, performing the checks for the conditions (1)-(3) in Theorem~\ref{theorem:subquiver}; or by a produced quiver being isomorphic to a quiver from the generated exchange graphs seeded by each of the non-mutation-acyclic quivers in the previous step.
    These checks classified the remaining 220 quivers, using mutation depths up to 12.

    Therefore, of the 667 isomorphism classes of rank 4 quivers with arrow weights 0, 1, or 2, the final class sizes (mutation-acyclic, non-mutation-acyclic) are (534, 133).
    These quivers are available on \href{https://github.com/KTKAW/MACHINE_LEARNING_MUTATION_ACYCLICITY_OF_QUIVERS.git}{\texttt{GitHub}}, listed as their respective adjacency matrices in canonical form. 
\end{proof}

\subsection{Machine Learning} \label{subsec:ml} \mbox{}\\
As a field, machine learning encompasses a broad range of computational techniques, often of statistical origin.
Within ML, this set of techniques can be subdivided into three core subfields: supervised, unsupervised, and reinforcement learning. 

Each subfield hosts methods suited to different data styles.
Specifically, supervised learning deals with paired data such that each input has a corresponding output, and the techniques then address problems of function fitting and approximation, fitting a general, highly parameterized function to map from input to output.
Unsupervised learning is in spirit data analysis, no output data is known for function fitting, so techniques are used to analyse the input dataset via methods of clustering and dimensional reduction.
The final subfield of reinforcement learning deals with input data that may be represented as a state space, where potential solutions to a problem can be evaluated with a score function, and then perturbed according to a set of actions.

The quiver data and problems considered here match most appropriately to methods in the first two subfields of supervised and unsupervised learning, and thus, in the following subsections, we introduce the techniques used from these areas in more detail.

\subsubsection{Supervised Learning}\label{sec:super}\mbox{}\\
Depending on whether the output data in each data pair comes from a finite set or not, a supervised learning problem can be categorised into either classification or regression.

In this work, the aim is to identify the mutation-acyclic property.
Therefore for input data as the adjacency matrix\footnote{A quiver's `adjacency matrix' is the same as its `exchange graph', the former terminology is more common in data science literature, whereas the latter is the term used within the cluster algebras field.} representation of a quiver, one can then assign a binary label to this quiver, which identifies whether it expresses the mutation-acyclic property or not.
Therefore, the dataset takes the form $(B,0)$ or $(B,1)$ where $B = B(Q)$ is the skew-symmetric adjacency matrix of a quiver, and the binary output label indicates whether the quiver is mutation-acyclic.
This output data is hence a finite set of size 2, making this problem a binary classification.

This dataset of pairs is partitioned into train data and test data, whereby the data is randomly shuffled, and 10\% of the data is partitioned off to be used to independently test the model's performance after training has completed.
Within the remaining train dataset, 30\% of this is then also partitioned off to form the validation set, which allows independent tracking of the performance throughout training (whereas the test data is only used after training is complete).
In practice, this process is generalised such that the dataset is partitioned 5 different ways (with the same split proportions but different shuffles of the full dataset prior to split), then 5 identically structured architectures are each trained and tested on one of these partitions.
This process is cross-validation and produces five performance measures, which can be averaged to improve statistical confidence in the reported learning performance.

Since this is a classification problem, performance is assessed based on how frequently the trained model correctly identifies the true output class for the input datapoint.
In this context, this is the ability of the model to identify whether the input quiver is mutation-acyclic or not.
The test performance is recorded in a confusion matrix: $C_{ij}$, from which the standard learning performance metrics are derived.
These metrics, along with their associated formulae, are provided in §\ref{app:ml_intro}.
Accuracy is used as the most standard interpretable metric, being the proportion of correctly classified test inputs, whilst Matthew's correlation coefficient (MCC) is used as a more stable, unbiased measure.
Both measures evaluate to 1 for perfect learning, whereas for no learning, accuracy evaluates to 0.5 (for balanced binary classes) and MCC to 0. 

During the training process, an optimiser algorithm will update the parameters of the model in order to minimise a loss function that evaluates the difference between the predicted and true output.
The standard loss function for binary classification problems is binary cross-entropy, as used here with the formula provided in §\ref{app:ml_intro}.

\paragraph{\textit{Neural Networks}}\mbox{}\\
Perhaps the most famous supervised learning architecture is that of neural networks (NNs).
Largely due to their amenability to a variety of problems across many fields.

NNs are built from an array of neurons, where each neuron represents the application of a linear and then a non-linear function.
For some input vector $\underline{x}$, the neuron acts as $\underline{x} \mapsto \phi(\underline{w}\cdot\underline{x} + b)$ for non-linear activation function $\phi$ and trainable parameters within the weight vector $\underline{w}$ and bias number $b$.
These neurons are then collected into layers, where every neuron in the layer receives the same input vector, being the concatenation of the all the previous layer's output numbers\footnote{The NN architecture described here is the prototypical dense feed-forward NN, which is quite general and reduces to other popular architectures such as Convolutional NNs on zeroing of certain parameters. 
Furthermore, more exotic architectures can be created by allowing layers to input to themselves or to previous layers creating cycles in the NN graph (developing into recurrent NNs and transformers).}.
To encourage function stability during the training process and reduce the chances of overfitting, each neuron's input from a previous neuron can be chosen to be set to zero with a fixed probability, known as the dropout probability.
The number of layers and number of neurons per layer are then hyperparameters which set the NN architecture, where the first layer neurons each receive the input vector (here the flattened adjacency matrix), and the output layer has as many neurons as the number of classes (for classification problems) and this layer's activation is set to a function which normalises the outputs to produce a probability distribution.
An example diagram of a NN is given in Figure \ref{fig:NN_diagram}.

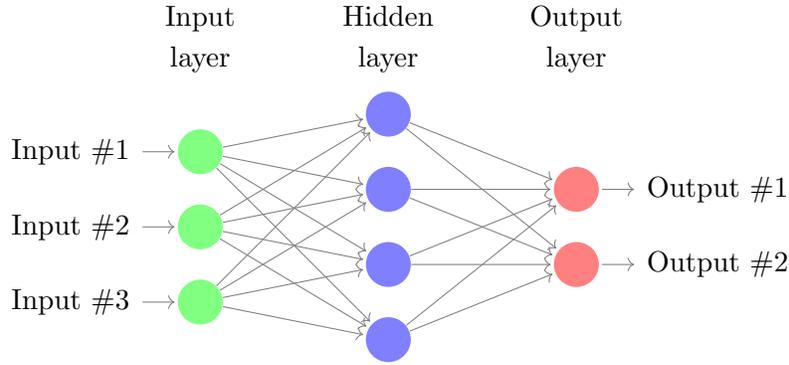
\begin{figure}[!t]
\centering
\begin{tikzpicture}[shorten >=1pt,->,draw=black!50, node distance=\layersep]
    \tikzstyle{every pin edge}=[<-,shorten <=1pt]
    \tikzstyle{neuron}=[circle,fill=black!25,minimum size=17pt,inner sep=0pt]
    \tikzstyle{input neuron}=[neuron, fill=green!50];
    \tikzstyle{output neuron}=[neuron, fill=red!50];
    \tikzstyle{hidden neuron}=[neuron, fill=blue!50];
    \tikzstyle{annot} = [text width=4em, text centered]
    \def\layersep{2.5cm}

    \foreach \name / \y in {1,...,3}
        \node[input neuron, pin=left:Input \#\y] (I-\y) at (0,-\y) {};

    \foreach \name / \y in {1,...,4}
        \path[yshift=0.5cm]
            node[hidden neuron] (H-\name) at (\layersep,-\y cm) {};

    \foreach \name / \y in {1,...,2}
        \path[yshift=-0.5cm]
            node[output neuron, pin={[pin edge={->}]right:Output \#\y}] (O-\name) at (2*\layersep,-\y cm) {};

    \foreach \source in {1,...,3}
        \foreach \dest in {1,...,4}
            \path (I-\source) edge (H-\dest);

    \foreach \source in {1,...,4}
        \foreach \dest in {1,...,2}
            \path (H-\source) edge (O-\dest);

    \node[annot,above of=H-1, node distance=1cm] (hl) {Hidden layer};
    \node[annot,left of=hl] {Input layer};
    \node[annot,right of=hl] {Output layer};
\end{tikzpicture}
\caption{An example neural network with one hidden layer; this can be generalised to more hidden layers and more neurons per layer. Each neuron represents the action of a linear and then a non-linear function on its input vector.}
\label{fig:NN_diagram}
\end{figure}

The choice of activation function determines the structure of the function approximation; the most common choice is the rectified linear unit (ReLU), which sets the full NN function to be piecewise linear.
There are generalisations of ReLU which smoothen the piecewise linear nature, including the scaled exponential linear unit (SELU), and the standard hyperbolic tangent function.
Finally, for classification problems, where output normalisation to a probability distribution is desired, the softmax activation function is often used on the final layer (as chosen for the architectures in this work).
This probability distribution over the output nodes is used in the loss for parameter updating, whilst for test classification, the output with the highest probability is selected as the predicted class for each test input datapoint.
The formulae for these activation functions are provided and described in §\ref{app:ml_intros_acts}.

Over the training process, the optimiser updates the parameters (weights and biases) according to some adaptation of stochastic gradient descent, here Adam \cite{kingma2017adam}.
The train dataset inputs are shuffled and split into batches of a prespecified size, then they are fed into the NN to produce test outputs which the loss is evaluated at after differentiating it with respect to the model parameters.
This process of iteratively updating parameters is repeated for all batches over the training dataset, completing one epoch.
After this, the train dataset is shuffled and resplit into batches, and the whole process is repeated, now iterating over a fixed number of prespecified epochs.
The step size in the parameter space taken by the optimiser with each update is scaled by a learning rate.
The batch size, number of epochs, and learning rate are further hyperparameters one can choose to define a NN architecture.

\paragraph{\textit{Support Vector Machines}}\mbox{}\\
Alternatively, perhaps the most optimised architectures for binary classification are support vector machines (SVMs).
Their interpretability is a significant advantage for distilling true mathematical insight from their learning, and a feature we aim to take advantage of in this study.

Considering vectorial input data of size $d$, one can imagine each input as a point in $\mathbb{R}^d$, with the dataset then populating this space.
For a binary classification problem, each datapoint is associated with a label indicating which of the two classes it belongs to.
The SVM architecture then seeks to find a linear codimension-1 hyperplane within $\mathbb{R}^d$ which best partitions the space such that either side of the plane corresponds to one of the classes.
The performance is measured by the number of points on the correct side of the hyperplane for their label, and then the distance of the hyperplane to the closest points on either side (which, when maximised, improves the stability of generalisation).

Figure \ref{fig:SVM_diagam} provides an example of a SVM fitting of 2-dimensional data, embedded in $\mathbb{R}^2$, such that the codimension-1 linear hyperplane is a straight line.
The data here is linearly separable, and hence, a hyperplane can be found that perfectly separates the data.
In fact, one can see that rotating the hyperplane slightly would still separate the data perfectly, and thus, for optimum generalisability, one seeks the maximum-margin hyperplane.
The margin is determined by orthogonally translating the hyperplane in both directions along the codimension-1 normal until the first point in each class is reached; the midpoint of this margin is the natural hyperplane position along its normal.
These first points are the \textit{support vectors} for the hyperplane, and completely determine it irrespective of the remaining datapoints.
The training process then seeks the hyperplane orientation that maximises this margin, as demonstrated by the shown hyperplane position in the figure.

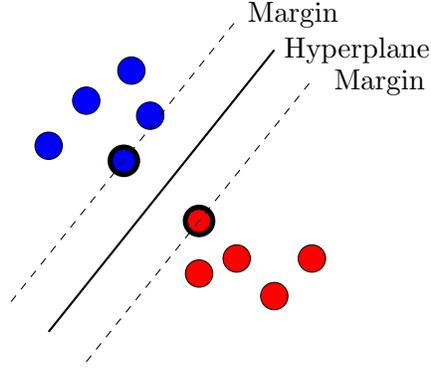
\begin{figure}[!t]
\centering
\begin{tikzpicture}
    \foreach \x/\y in {0.5/2.5, 1/3.1, 1.5/2.3, 1.6/3.5, 1.85/2.9} {
        \node[draw, circle, fill=blue] at (\x, \y) {};
    }
    \foreach \x/\y in {2.5/0.8, 3/1, 2.5/1.5, 3.5/0.5, 4/1} {
        \node[draw, circle, fill=red] at (\x, \y) {};
    }

    \node[draw, circle, fill=none, line width=0.7mm] at (1.5, 2.3) {};
    \node[draw, circle, fill=none, line width=0.7mm] at (2.5, 1.5) {};

    \draw[thick] (0.5,0.025) -- (3.5,3.775); 

    \draw[dashed] (0,0.425) -- (3,4.175);
    \draw[dashed] (1,-0.375) -- (4,3.375);

    \node at (4.6, 3.775) {Hyperplane};
    \node at (3.75, 4.25) {Margin};
    \node at (4.9, 3.35) {Margin};

\end{tikzpicture}
\caption{An example support vector machine in 2-dimensions. The datapoints are coloured according to the class they belong to (either blue or red), and the support vectors used to construct the codimension-1 supporting hyperplane with maximum margin are bolded. Here, the hyperplane is linear, but the kernel trick could be used to construct higher-degree hyperplane equations.}
\label{fig:SVM_diagam}
\end{figure}

In reality, it is unlikely a generic dataset will be linearly separable; in that case, the hinge loss is used in place of it, defined as
\begin{equation}\label{eq:hinge}
        \text{Hinge Loss} \vcentcolon = \sqrt{\Sigma_i w_i^2} + C \bigg(\frac{1}{N}\sum_k \bigg(\text{max}\big(0,1-y_k\big(\Sigma_i(w_ix_{i,k})-b\big)\big)\bigg)\bigg)\;
\end{equation}
for hyperplane equation $0=\Sigma_i(w_ix_i) - b$ on input vectors $x_{i,k}$ with class labels $y_k$, where index $i$ runs over the vector dimensions, and $k$ the index of each (input vector, output label) in the dataset (of which there are $N$).
The hyperplane equation parameters $(w_i,b)$ are learned in the training process, where $2b/\sqrt{\Sigma_i w_i^2}$ determines the margin size.
The use of the $\text{max}(\cdot )$ function means data points that lie on the outside of the margins (with equations $\Sigma_i (w_ix_i)-b = \pm 1$) do not contribute to the loss as they are already correctly classified.
The final parameter $C$ dictates the regularisation, where this is large the main part of the loss function dominates (the soft-margin loss) and the optimiser prioritises a model which fits well with many datapoints correctly classified; conversely where $C$ is low the regularisation first term in \eqref{eq:hinge} dominates prioritising a simpler hyperplane equation with less parameters.
Since the goal in the application of this architecture is interpretability, a low value of $C=1$ is used.

Overall, the SVM parameters $(w_i,b)$ are optimised throughout the training process in an iterative manner in order to minimise this hinge loss.
The dataset is again split into train and test data, where the test data is reserved for use after training to evaluate performance.

This implementation of SVMs seeks an optimal linear hyperplane; however, many datasets are not linearly separable, where an optimal non-linear hyperplane would be preferable.
To facilitate the SVM process accommodating non-linear hyperplane equations, the `kernel trick' is used.
Here, the non-linear terms in the hyperplane equations are considered as additional dimensions, effectively trading the non-linear degrees of freedom for linear degrees of freedom in a higher-dimensional space.
A linear hyperplane in this higher-dimensional space can then be mapped back to a non-linear hyperplane in the original space, and thus the ideology and process are very much the same.

The beauty in the trick is that the higher-dimensional embedding of each input is not required explicitly, saving substantial computational cost. 
Where the similarity between any two points in the higher-dimensional space is required for the training process, the kernel can compute this directly as a function of the original input vectors.
For this application, we are primarily concerned with interpretability and stick to the simplest non-linear kernels, which are polynomial.
For the fitting of a degree $\delta$ polynomial hyperplane, the kernel function form is:
\begin{equation}\label{eq:kernel}
    \text{Ker}(\hat{x}_{k_1},\hat{x}_{k_2}) \vcentcolon = (\gamma (\hat{x}_{k_1} \cdot \hat{x}_{k_2}) + \beta)^\delta\;,
\end{equation}
acting on input vectors $\hat{x}_{k_1}$ or $\hat{x}_{k_2}$ (where the dimension index is now implicit in the $\hat{\cdot}$ ), using additional input parameters $(\gamma,\beta)$ to optionally offset the hyperplane, here we stick to the default\footnote{Note for a linear kernel with $\delta=1$ this kernel reduces to the simple dot product in the original input space (for $\gamma=1,\beta=0$).} initialisations $\gamma=1/(\sigma^2d)$, $\beta=0$; for data dimension $d$ and data variance $\sigma^2$.

Once a SVM has been trained, the hyperplane equation is given by
\begin{equation}
    0 = \sum_{\tilde{k}} y_{\tilde{k}} \alpha_{\tilde{k}} \text{Ker}(\hat{x}_{\tilde{k}},\hat{x})+b\;,
\end{equation}
which can be evaluated for any input vector $\hat{x}$, and is based on the support vectors $\hat{x}_{\tilde{k}}$ from the dataset (i.e. $\tilde{k}$ does not run over all dataset vectors, just those which are support vectors) with labels $y_{\tilde{k}}$, explicitly using the kernel function for the specified degree as in \eqref{eq:kernel}, and the trained parameters\footnote{These are alike the $(w_i,b)$ in the linear case.} $(\alpha_{\tilde{k}},b)$. 

Using these supervised techniques, we aim to identify the mutation-acyclicity property in generic input quivers, particularly extracting full SVM equations to initiate interpretation.

\subsubsection{Unsupervised Learning}\label{sec:unsuper}\mbox{}\\
Depending on the goal of the data analysis, most unsupervised learning techniques can be applied for either dimensional reduction or clustering.

Prior to the application of supervised methods, it can be prudent to identify the optimal representation for the input data.
In this work, we address this goal via the prototypical method of dimensional reduction on the quiver adjacency matrix data: \textit{Principal Component Analysis} (PCA).

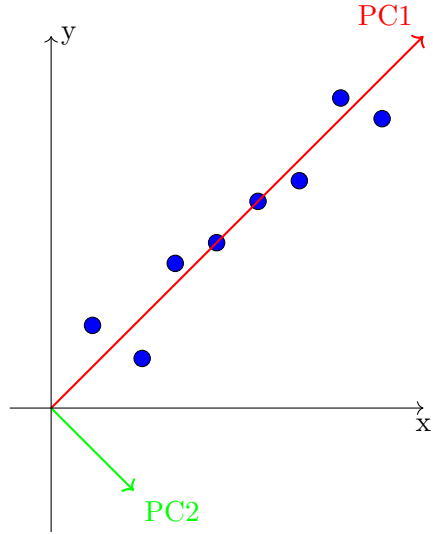
\begin{figure}[!t]
\centering
\begin{tikzpicture}[scale=0.55]

    \foreach \x/\y in {1/2, 2.2/1.2, 3/3.5, 4/4, 5/5, 6/5.5, 7/7.5, 8/7} {
        \node[draw, circle, fill=blue, scale=0.6] at (\x, \y) {};
    }

    \draw[thick, red, ->] (0,0) -- (9,9) node[anchor=south east] {PC1};
    \draw[thick, green, ->] (0,0) -- (2,-2) node[anchor=north west] {PC2};

    \draw[->] (-1,0) -- (9,0) node[anchor=north] {x};
    \draw[->] (0,-3) -- (0,9) node[anchor=west] {y};

\end{tikzpicture}
\caption{An example of principal component analysis in 2-dimensions. The datapoints are coloured in blue, and after diagonalisation of the covariance matrix, the principal components (scaled by their respective eigenvalues) are also shown on the plot. From their respective sizes, projecting all datapoints onto PC1 would be an optimal 1-dimensional linear projection.}
\label{fig:PCA_diagram}
\end{figure}

As a method, PCA utilises the diagonalisation of the dataset's covariance matrix, where for the input dataset $\{\underline{x}\}$, the covariance $K_{ij}$ is computed as
\begin{equation}
    K_{ij} = E(x_i - E(x_i))E(x_j - E(x_j))\;,
\end{equation}
for $x_i$ the $i$th component of the input vectors $\underline{x}$ and $E( \cdot )$ the expectation value over the dataset.
This matrix gives a centralised measure of second-order moments across the input components, and its diagonalisation produces eigenvectors, in this context called principal components, which determine a natural orthogonal basis for the data variation.
Since the covariance matrix is symmetric, the eigenvalues are all real\footnote{Actually, the covariance matrix is positive semi-definite, so the eigenvalues are also all non-negative.}, and provide a means of determining which principal components are most informative on the data's variation.
Hence, for $d$-dimensional input vectors, the $d$ principal components can be sorted in decreasing size of their eigenvalues.

Therefore, to perform a dimensional reduction from the $d$-dimensional input onto some linearly optimum $d'$-dimensional subspace, one may project the dataset using a projection matrix formed from the first $d'$ normalised principal components.
The proportion of the data variation explained by this representation may be determined by the sum of the respective $d'$ normalised eigenvalues\footnote{Each normalised eigenvalue is often called the explained variance ratio. Here normalisation is across the spectrum of eigenvalues, whereas for each eigenvector it is within each vector.}.
A diagram demonstrating PCA decomposition of a 2-dimensional dataset is shown in Figure \ref{fig:PCA_diagram}, a linearly\footnote{Note that PCA can be generalised to non-linear bases, also using kernels. However, here for this initial exploratory analysis, we focus on the simplest linear kernel.} optimum 1-dimensional projection would be along the first principal component.

\section{Data Generation}\label{sec:data}
To study the properties of quivers, we must generate datasets that can be easily fed into statistical learning algorithms.
In particular, recall that the operation of quiver mutation can be performed iteratively; thus, a starting point would be to generate datasets from exhaustively applying the mutation to a starting set of quivers.  
Methods of statistical learning are predominantly written in the language of linear algebra; hence, we must first express our quivers in this language. 
To this end, we represent every quiver in our dataset in terms of its exchange matrix $b_{ij}$, which for a rank $n$ quiver has size $n\times n$. 
For computational convenience, we then flatten the components of our exchange matrix to a $1 \times n^2$ dimensional vector.
In this paper, we restrict ourselves to rank $4$ quivers, which correspondingly means that our dataset will be comprised of $1 \times 16 $ dimensional vectors. 

\begin{ex} The quiver in Figure \ref{A41} has adjacency matrix
\begin{equation}
    \begin{pmatrix}
        0 & 1 & 0 & 0 \\ -1 & 0 & 1 & 0 \\ 0 & -1 & 0 & 1 \\ 0 & 0 & -1 & 0
    \end{pmatrix}\;,
\end{equation}
flattened to the input vector: $(0,1,0,0,-1,0,1,0,0,-1,0,1,0,0,-1,0)$. The respective output class label then depends on the investigation problem being learnt.
\end{ex}

For the three NN investigations carried out, we generate three datasets through this exhaustive mutation process: Dataset 1, Dataset 2, and Dataset 3. 

Dataset 1 was generated from four cluster algebra mutation classes of quivers, named here: A4-like, D4-like, Non-Mutation-Acyclic-1 (NMA1), and Non-Mutation-Acyclic-2 (NMA2), whose initial quivers (defining the cluster algebra) are respectively illustrated in Figures \ref{A4_GRAPHS}, \ref{D4_GRAPHS}, and \ref{NMA_GRAPHS}.
For each initial class, we mutate the quivers to depth 8. The number of entries for each class (as well as the class labelling associated with each class) is given in Table \ref{DATASET1TABLE}. 

Dataset 2 was similarly sourced from mutation classes with initial quivers depicted in Figures \ref{A4_GRAPHS}, \ref{D4_GRAPHS}, and \ref{NMA_GRAPHS}. In this instance, we mutate each of the initial quivers up to depth 6. We then combine the A4-Like and D4-Like quivers into a single shuffled input class of mutation-acyclic (MA) quivers, and the NMA1 and NMA2 into a single shuffled input class of non-mutation-acyclic (NMA) quivers. The number of entries and class label is described in Table \ref{DATASET2TABLE}. 

Dataset 3 is a larger equivalent of dataset 2, and is sourced from the exhaustive mutation of an initial set of mutation-acyclic quivers and non-mutation-acyclic quivers (see Table \ref{DATASET3TABLE}). For the mutation-acyclic class, we initially start off with a set of A4-like and D4-like quivers (as seen in \ref{A4_GRAPHS} and \ref{D4_GRAPHS}, explicitly acyclic) with either $1$, $2$, $3$, and $4$ edges between the vertices. We then remove any isomorphic quivers from this initial set. To generate the full mutation set, we mutate each quiver in the initial quiver set up to mutation depth $7$. Conversely, to generate the non-mutation-acyclic class, initial quivers given in Figures \ref{BOX_GRAPHS}, 
\ref{NMA_LIKE_S1_GRAPHS}, and \ref{NMA_LIKE_S2_GRAPHS} are used; with all the edges not a part of the $3$-cycle in Figure \ref{NMA_LIKE_S1_GRAPHS} having lengths between $0$ and $4$. Then, removing any graphs that are not connected, and removing any isomorphic quivers. We then mutate the dataset up to depth $5$.

The final investigation uses the more interpretable SVM architecture on a fourth dataset: Dataset 4. This dataset took the 667 rank 4 quivers with arrow weights $\{0,1,2\}$, classified according to being mutation-acyclic or non-mutation-acyclic via the proof of Theorem \ref{ca_theorem}, and generated their full isomorphism classes (counts presented in Table \ref{DATASET4TABLE}).

Emphasising that the exponential nature of the mutation process makes data generation very unpredictable, where different mutation classes can have significantly different sizes at the same depth. 
The depth choice was hence motivated by a goal to generate sufficiently many quivers per class, taken to be $\sim10000$, and then consistently use that depth across the classes in an investigation.  
This is why different depths were used between investigation datasets.
These datasets, as well as the generation scripts, are available at this paper's complementary \href{https://github.com/KTKAW/MACHINE_LEARNING_MUTATION_ACYCLICITY_OF_QUIVERS.git}{\texttt{GitHub}} repository.

\begin{table}[h!]
\centering
\begin{tabular}{|l|c|c|}
\hline
\textbf{Class} & \multicolumn{1}{l|}{\textbf{Number of Matrices}} & \multicolumn{1}{l|}{\textbf{Class Label}} \\ \hline
A4-Like        & 39363                                         & 0                                           \\ \hline
D4-Like        & 26242                                         & 1                                           \\ \hline
NMA1            & 13121                                         & 2                                           \\ \hline
NMA2            & 13121                                         & 3                                           \\ \hline
\end{tabular}
\caption{Dataset 1 Information}
\label{DATASET1TABLE}
\end{table}

\begin{table}[h!]
\centering
\begin{tabular}{|l|c|c|}
\hline
\textbf{Class} & \multicolumn{1}{l|}{\textbf{Number of Matrices}} & \multicolumn{1}{l|}{\textbf{Class Label}} \\ \hline
MA             & 7285                                          & 0                                           \\ \hline
NMA            & 2914                                          & 1                                           \\ \hline
\end{tabular}
\caption{Dataset 2 Information}
\label{DATASET2TABLE}
\end{table}

\begin{table}[h!]
\centering
\begin{tabular}{|l|c|c|}
\hline
\textbf{Class} & \multicolumn{1}{l|}{\textbf{Number of Matrices}} & \multicolumn{1}{l|}{\textbf{Class Label}} \\ \hline
MA             & 236142                                       & 0                                           \\ \hline
NMA            & 184785                                          & 1                                           \\ \hline
\end{tabular}
\caption{Dataset 3 Information}
\label{DATASET3TABLE}
\end{table}

\begin{table}[h!]
\centering
\begin{tabular}{|l|c|}
\hline
\textbf{Class} & \multicolumn{1}{l|}{\textbf{Number of Matrices}} \\ \hline
MA             & 12082                                     \\ \hline
NMA            & 3022                                          \\ \hline
\end{tabular}
\caption{Dataset 4 Information}
\label{DATASET4TABLE}
\end{table}

\newpage
\section{Results}\label{sec:ml}
Using the datasets described in §\ref{sec:data}, machine learning methods are now implemented to learn the mutation-acyclic property for input quivers. 
Since quivers can be mutation-acyclic or not, this is, in general, a problem of binary classification; thus, popular classification architectures are selected for experimentation: the ubiquitous NN architecture and the more interpretable SVM architecture.

\subsection{Neural Networks: Predicting Mutation-Acyclicity}\label{sec:nn}\mbox{}\\
As popular architectures for classification, NNs here are implemented for the task of identifying whether a generic input quiver is mutation-acyclic or not; a property extremely difficult to discern by eye (and generally analytically unknown for rank $>3$ cases).

In doing so, three experiments are carried out, each a classification between classes of quivers sourced from a variety of cluster algebras, these datasets were described in §\ref{sec:data}, and generated from initial quivers depicted in §\ref{app:quivers}.
The experiments start by classifying between mutation classes, then between superclasses of many mutation classes chosen such that all quivers in the superclass are either mutation-acyclic or non-mutation-acyclic.
The performance results are subsequently described.

\subsubsection{Classifying Mutation Classes}\label{sec:exp1}\mbox{}\\
As a warm-up exercise, this first experiment explores how effectively a NN can distinguish four different mutation classes of quivers, which we denote: A4-like, D4-like, Non-Mutation-Acyclic-1 (NMA1), and Non-Mutation-Acyclic-2 (NMA2); all contained within Dataset 1.
This is similar in style to the NN classification in \cite{Dechant:2022ccf}, albeit expanded to new mutation classes.

We utilize a fully connected deep neural network with $128$ nodes in each layer with either a SELU, ReLU, or tanh activation function (NN structure given in Figure \ref{First_Experiment_Class_NN}, established via heuristic hyperparameter tuning).
Between each layer, we use a dropout rate of $0.4$ to prevent overfitting. 
The final layer of the model had $4$ nodes, which are activated by a softmax function, to output a normalised probability distribution over the classes.
Dataset 1 is split into a train:test split of 90:10, with training data being split into a 70:30 split to accommodate the creation of a smaller validation training set. 
The model was trained over $10000$ epochs with a learning rate of $0.0001$ (using the Adam optimizer). In addition, each epoch utilized batching, using batch sizes of $200$. The model was trained to minimize the Sparse Categorical Cross-entropy loss function. 
In addition, the dataset was balanced such that there was an equal distribution of each class in the dataset (for the smallest class size $s$, uniformly sample $s$ quivers from each other class).

The model had an accuracy on test data of $61.8 \%$, with a MCC of $0.507$.
These performances represent notable learning, since no ability to identify the cluster algebra from an input quiver is represented by (accuracy, MCC) scores of (25\%, 0) respectively, which both of these substantially exceed.
Since the quiver data between classes is impossible to classify by eye, and largely impractical to do analytically via mutation, these methods already display practical use in aiding the identification of cluster algebras for input quivers.
These performance scores set a useful baseline for comparison in the subsequent experiments.

\subsubsection{Classifying Mutation-Acyclicity: Few Algebras}\label{sec:exp2}\mbox{}\\
In the second experiment, we tested how effective a NN was at determining the difference between mutation-acyclic quivers and non-mutation-acyclic quivers at rank 4. 
This experiment used Dataset 2 as its input, where two mutation classes of each type (mutation-acyclic or non-mutation-acyclic) were merged to form a larger database, causing the architecture to focus on the mutation-acyclicity property alone, instead of just the mutation class.
Again, the data was class-balanced to ensure an equal distribution.
The NN trained for this experiment is almost identical in structure to that used for the first experiment, with the only difference being that we have used $2$ nodes in the output layer as opposed to $4$ in experiment 1 of §\ref{sec:exp1}. Furthermore, the test:train split, validation split, learning rate, batch size, number of epochs, and optimizer are the same as those for experiment 1. 

\begin{figure}[!t]
    \centering
    \includegraphics[width =0.7\textwidth ]{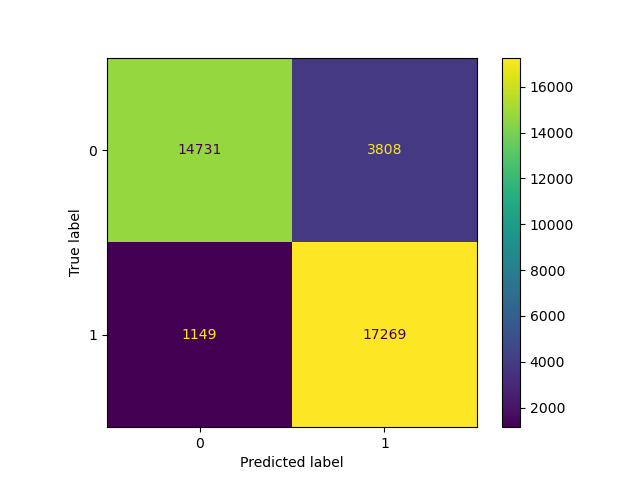}
    \caption{Confusion matrix for predictions on Dataset 3's randomly sampled test set averaged across the five cross-validation runs for mutation acyclic quivers (0) and non-mutation acyclic quivers (1).}
    \label{Fig:Con_Mat_EXP_3}
\end{figure} 

The model had an accuracy on test data of $87.7 \%$, with a MCC of $0.764$.
Null learning for two outputs leads to (accuracy, MCC) performance scores of (50\%, 0), and again, both the performances substantially exceed these values, demonstrating notable learning.
The more informative MCC measure is particularly impressive, far exceeding 0, and indicating that the mutation-acyclicity property can be well identified for this database of quivers.
This further hints at the existence of a combinatorial invariant of the quiver input data, which the NN function is approximating to then use for this classification.

\subsubsection{Classifying Mutation-Acyclicity: Many Algebras}\label{sec:exp3}\mbox{}\\
Following the promising results in §\ref{sec:exp2}, experiment 3 aimed to probe how the NN's ability to classify mutation-acyclic quivers and non-mutation-acyclic quivers would change under both a large increase in the data size and expanding the superclasses to include more cluster algebra mutation classes. To this end, we train the NN on the larger Dataset 3, where many more mutation classes of each type were used to form the classified mutation-acyclic/non-mutation-acyclic classes. 

Owing to the increased size of the dataset, we adopt a shallower\footnote{Since for this experiment we are primarily interested in the efficacy of the NN methods, the specific architecture is less important. Additionally, some experimentation with the previous architecture during hyperparameter tuning showed negligible performance differences despite longer training times.} NN architecture to the previous two experiments (as shown in Figure \ref{FIRST_EXP_NN_DIAGRAM_BIG_DATA}). 
We utilize a fully connected deep neural network with $128$ nodes in each layer with either a SELU or ReLU activation function.
Between each layer, we use a dropout of $0.3$ to avoid over-fitting. 
The final layer of the model had $2$ nodes, which are activated by a softmax function.
Dataset 3 is split into a train:test split of 90:10, with training data being further split into a 70:30 split to accommodate the creation of a smaller validation training set. 
The model was trained over $10000$ epochs with a learning rate of $0.0001$ (using the Adam optimizer). In addition, each epoch utilized batching, using batch sizes of $100$. The model was again trained to minimize the Sparse Categorical Cross-entropy function. 
In addition, the dataset was balanced such that there was an equal distribution of each class in the dataset.

As the final NN experiment, to provide confidence in the learning performances, cross-validation was implemented, such that the experiment was run 5 times, with 5 independently initialised NNs, using random shuffling of the data before splitting and training within each run. The results for test accuracy and MCC across the runs, as well as the average, are outlined in Table \ref{DATASET3RESULTSTABLE}. In addition, an averaged confusion matrix for predictions on the test set across all 5 runs is given in Figure \ref{Fig:Con_Mat_EXP_3}. 

\begin{table}[!t]
\centering
\begin{tabular}{|c|c|c|}
\hline
\textbf{Run} & \textbf{MCC} & \textbf{Test Accuracy Percentage} \\ \hline
1                & 0.731            & 86.3\%                 \\ \hline 
2                & 0.754            & 87.4\%                 \\ \hline 
3                & 0.745             & 86.9\%                 \\ \hline 
4                & 0.748           & 87.1\%                 \\ \hline 
5                & 0.759           & 87.6\%                 \\ \hline 
Average                & 0.747 $\pm$ 0.004           & 87.0 $\pm$ 
0.2 \%                \\ \hline 
\end{tabular}
\caption{Results for Dataset 3 Training across the five cross-validation runs, displaying also the average performance scores with standard error.}
\label{DATASET3RESULTSTABLE}
\end{table}

The results shown here express equally high performances to §\ref{sec:exp2}.
They are also consistent across the cross-validation runs, providing statistical confidence in the average performance scores.
Since the average MCC score far exceeds null performance ($0.747 \gg 0$, when MCC $\in [-1,1]$), this strongly supports the existence of a mutation-invariant combinatorial constant that the NNs are approximating, which dictates the mutation-acyclicity property, analogous to the Markov constant for rank 3 quivers.
To open the door to its likely highly complex structure, we turn to a more interpretable ML method: SVMs, as outlined in the subsequent section.

\subsection{Support Vector Machines: Identifying Separating Hyperplanes}\label{sec:sv} \mbox{}\\
The results of §\ref{sec:nn} show promising performances, indicating there is likely some underlying structure in the edge multiplicity combinatorics that these highly parameterised NN models can leverage.
However, as highly non-linear models, they are difficult to interpret and to extract mathematical insight from.

In this section, we shift focus to a complementary data setup (i.e., splitting into mutation-acyclic and non-mutation-acyclic), but now considering all possible quivers with an edge multiplicity $<3$.
This is Dataset 4, described in §\ref{sec:data}, from the proof of Theorem \ref{ca_theorem}.
Additionally, we use a more interpretable ML architecture: SVMs (as introduced in §\ref{sec:super}), with the intention that the simpler design and directly extractable equations will be more insightful to mathematicians seeking to design appropriate general invariants of mutation-acyclicity.

\subsubsection{Principal Component Analysis}\mbox{}\\
To develop an analytic understanding of the separation between NMA and MA  quivers, it is advantageous to first determine how each type of quiver is distributed within the dataset (Dataset 4). 
Recalling that each quiver is represented by a  $1\times16$ vector, we see that our dataset is defined within a $16$-dimensional space, which may be taken to be $\mathbb{R}^{16}$. 
For computational efficiency, it would be effective to probe which of the coordinate directions the dataset is either aligned with or most spread across. In addition, it would also be prudent to identify any patterns or correlations between the various coordinate directions. The overall aim of this procedure would be to find which directions are most significant and then project our dataset to only depend on these favorable coordinate directions.
PCA, as introduced in §\ref{sec:unsuper}, provides a useful tool in this instance to determine which of $16$ dimensions are the most pivotal in describing the overall nature of the data, and what efficient compressed representations exist. 

In performing PCA (with a linear kernel) on the dataset, we allow for the choice of a generic basis for which the dataset is transformed (which is an artifact of the algorithm e.g \textit{depending on the order of the dataset}) for this analysis. The eigenvalues of the PCA decomposition for this new generic basis are given in Table \ref{tab:PCA}. 
From Table \ref{tab:PCA}, it becomes clear that the first 6 dimensions (which share the \textit{same} eigenvalue) are the most important for our analysis, as they are by far the largest eigenvalues we have for all $16$ dimensions. 

In Figure \ref{PCA_DISPLAY}, we plot the spread of the dataset pairwise between each of these $6$ dimensions.
The above results imply that we can form an eigenbasis from the $6$ most important directions, onto which we can project our dataset, to reduce its dimensionality (from $16$ to $6$).

\begin{longtable}[!t]{|c|c|c|}
\hline
\textbf{}                        & \textbf{Eigenvalue}          \\ \hline
\multicolumn{1}{|c|}{\textbf{1}} & \textit{0.16666666666666674} \\ \hline
\multicolumn{1}{|c|}{\textbf{2}} & \textit{0.16666666666666674} \\ \hline
\multicolumn{1}{|c|}{\textbf{3}} & \textit{0.16666666666666669} \\ \hline
\multicolumn{1}{|c|}{\textbf{4}} & \textit{0.16666666666666669} \\ \hline
\multicolumn{1}{|c|}{\textbf{5}} & \textit{0.16666666666666663} \\ \hline
\multicolumn{1}{|c|}{\textbf{6}} & \textit{0.16666666666666657} \\ \hline
\textbf{7,8,...,16}                       & 0       \\ \hline
\caption{Eigenvalues for PCA transformed rank 4 quiver dataset (Dataset 4), those $<10^{-30}$ have been rounded to $0$.} 
\label{tab:PCA}
\end{longtable}

\begin{figure}[!t]
    \centering
    \includegraphics[width =\textwidth ]{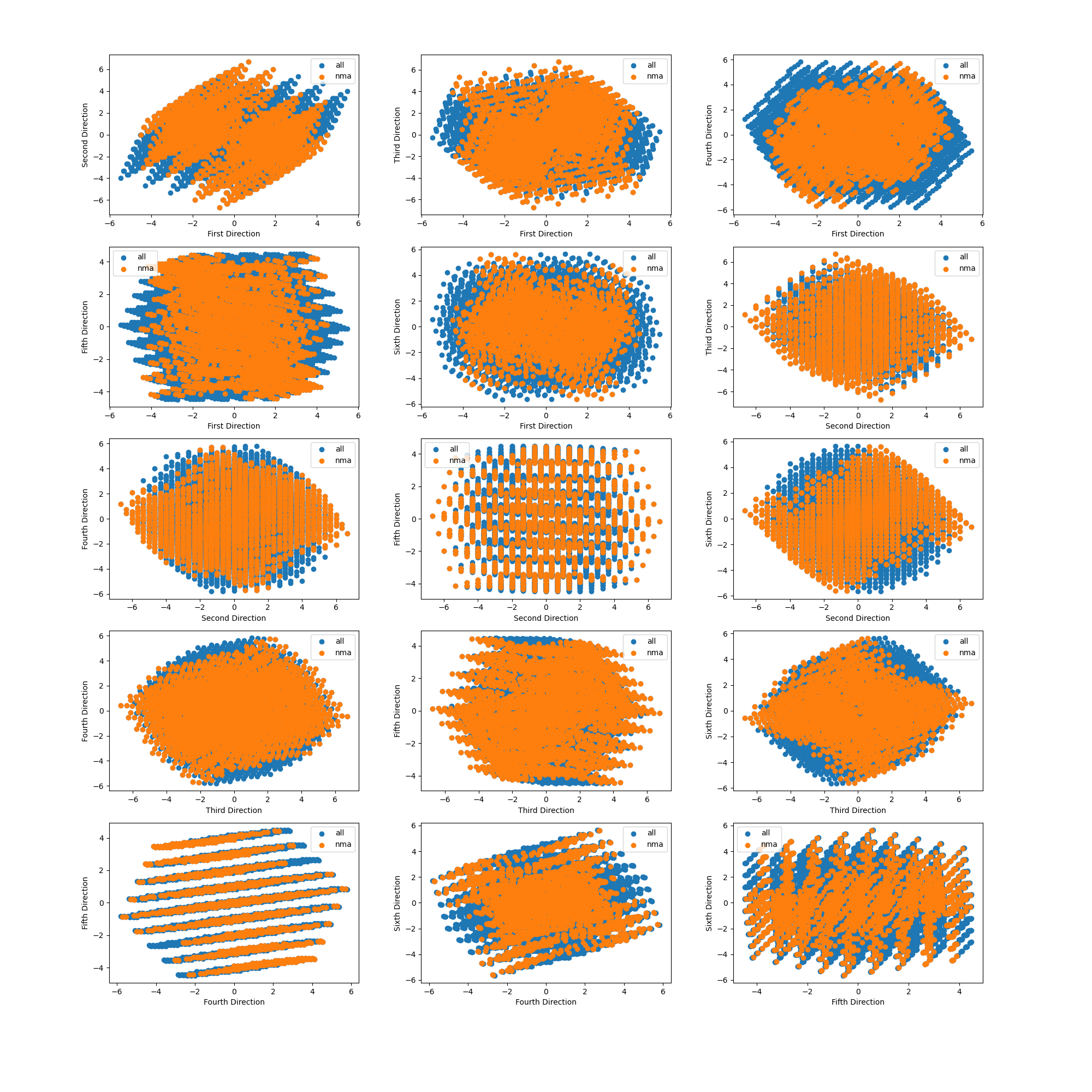}
    \caption{Spread of non-mutation-acyclic quivers (orange dots) vs mutation-acyclic quivers (blue dots) for pairwise projections of the 6 principal directions chosen from PCA of Dataset 4.}
    \label{PCA_DISPLAY}
\end{figure} 

However, before we proceed any further, we observe that the exchange matrix for a rank $4$ quiver is:
\begin{equation}
b_{ij}=
\begin{pmatrix}
0 & e_{0}& e_{1}&e_{2}\\ 
-e_{0} & 0 & e_{3} & e_{4} \\ 
 -e_{1} & -e_{3} & 0 & e_{5}\\
 -e_{2} & -e_{4} & -e_{5} & 0
\end{pmatrix},
\label{equ:exchange_matrix}
\end{equation}

due to the representation's antisymmetry, where $e_{0}$ through to $e_{5}$ are integers. 
Therefore, these input exchange matrices have 6 degrees of freedom, and can naturally be represented in a 6-dimensional space as the vector: $(e_0,e_1,e_2,e_3,e_4,e_5)$.

The PCA eigenvalues nicely corroborate this 6-dimensional representation, confirming that no further sensible linear compression is possible, and since the method is linear, the PCA eigenvectors are some linear combination of these values.
Therefore, motivated by the PCA, the next investigations will use this lossless compressed 6-dimensional representation as the quiver input. 

\subsubsection{SVM Analysis}\mbox{}\\
With a more natural and efficient representation of the input quiver data, we can now train a SVM, as introduced in §\ref{sec:super}, on the dataset to identify a $5$-dimensional (codimension-1) separating hyperplane between the mutation-acyclic (MA) and non-mutation-acyclic (NMA) quivers in our dataset. To maximise interpretability, the SVM analysis was initially applied with a linear kernel, then generalised with polynomial kernels of increasing degree, covering orders $1$ to $26$.
In order to classify the dataset (of Dataset 4) into either MA or NMA, we assigned MA quivers a label of $-1$, and NMA quivers the label of $1$. Moreover, we assigned class weights, with MA quivers having a weight of $1$ and NMA quivers having a weight of $3.998$. 
This was done to reflect that NMA quivers had a smaller population than MA in the dataset, and this weighting for SVM training exactly inverts this proportion such that both classes are viewed with equal importance in the training loss. Additionally, the regularisation parameter $C$ was set to $1$, leading to a higher regularisation and prioritising identifying a simpler, more interpretable hyperplane.

We outline the learning results for each kernel order on the data in Figure \ref{PCA_MC} and Table  \ref{tab:SVMRAW}. Here, we take MCC to be the test accuracy measure, as it is better adapted to the uneven class sizes. 

\begin{figure}[!t]
    \centering
    \includegraphics[width=0.7\textwidth]{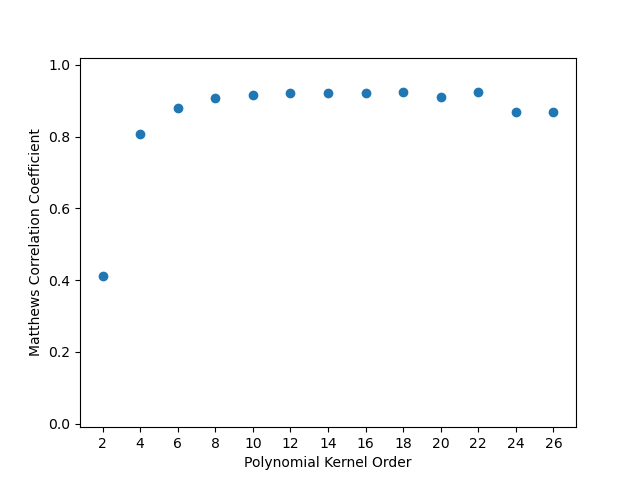}
    \caption{Final MCC scores for Support Vector Machines with even polynomial kernels of increasing degree, trained on the compressed 6-dimensional input for rank 4 quivers with edge multiplicity $<3$.}
    \label{PCA_MC} 
\end{figure}

From Table \ref{tab:SVMRAW}, we see that odd ordered polynomials kernels prove completely ineffective in separating NMA from MA; likely since the data naturally contains all quivers in each relabelling isomorphism orbit, so is naturally symmetric, which the odd degree functions cannot accommodate.
However, even ordered polynomial kernels prove more successful in separating the classes. In particular, the order $18$ polynomial from the degree 18 kernel exhibited the highest classification performance between NMA and MA quivers, with a MCC of 0.924.  

The MCC performance results for each polynomial kernel are listed in Table \ref{tab:SVMRAW}, with the number of terms in each polynomial equation. 
Additionally, the analytic equations for the identified optimally-classifying SVM hyperplanes are available at the paper's \href{https://github.com/KTKAW/MACHINE_LEARNING_MUTATION_ACYCLICITY_OF_QUIVERS.git}{\texttt{GitHub}}.

At this juncture, we also note that additional SVM training was performed on the so-called canonical form of Dataset 4 (isomorphism classes). This resulted in, at face value, favourable training results (an MCC score of $0.97$ at SVM polynomial kernel order $12$); however, the number of terms in the resultant classification polynomials exceeded the number of degrees of freedom for the system; hence, they are not reported here as memorisation schemes are likely occurring.

The equation determined from the order $6$ polynomial is the most interesting to us, as it exhibits excellent prediction despite a low kernel degree and number of terms\footnote{The regularisation contribution could have been increased to discourage superfluous terms in the SVM equation, however this shifts the learning goal away from the core problem and the interpretation of the near-zero coefficients of non-linear basis monomials is somewhat ambiguous. Hence, focus on a smaller, more interpretable basis is desirable.}.
Since Dataset 4 contains 15104 quivers, and respectively 90624 input integers, despite these equations being long (where the degree 6 equation has 464 terms), this is still orders of magnitude less degrees of freedom than the input dataset; such that the SVMs are consolidating the mutation-acyclic property quite significantly.
The existence of an equation that can classify so well solidifies this motivation behind the existence of a mutation invariant for rank 4 quivers, which determines the mutation-acyclicity property.
Despite this degree 6 equation perhaps still being over-parameterised, the explicit nature of the equation's representation causes it to be a combinatorial invariant, with the hope that the equation provides some hints towards a more intuitive construction.
{
\begin{table}[!t]
\centering
\begin{tabular}{|c|c|c|c|}
\hline
\multicolumn{1}{|l|}{\textbf{Polynomial Kernel Degree}} & \textbf{MCC} & \multicolumn{1}{c|}{\textbf{Number of Terms}} \\ \hline
\textbf{1}                                       & \textit{0.000}                  & 8                                                                            \\ \hline 
\textbf{2}                                       & \textit{0.410}   & 22                                                                                    \\ \hline 
\textbf{3}                                       & \textit{0.000}                  & 58                                                                            \\ \hline 
\textbf{4}                                       & \textit{0.807}   & 128                                                                                 \\ \hline 
\textbf{5}                                       & \textit{0.000}                  & 254                                                                                
\\ \hline 
\textbf{6}                                       & \textit{0.879}                  & 464     
 
\\ \hline 
\textbf{8}                                       & \textit{0.907}                  &     1288                                                                               \\ \hline 
\textbf{10}                                       & \textit{0.916}                  & 3004                                                                                 \\ \hline 
\textbf{12}                                       & \textit{0.921}                  & 6189                                                    
\\ \hline 
\textbf{14}                                       & \textit{0.921}                  & 11629                                                                               
\\ \hline 
\textbf{16}                                       & \textit{0.921}                  & 20350                                                                                  \\ \hline 
\textbf{18}                                       & \textit{0.924} 
 & 33650
 
\\ \hline 
\textbf{20}                                       & \textit{0.909} 
 & 53132

 \\ \hline 
\textbf{22}                                       & \textit{0.924} 
 & 80740
  
 \\ \hline 
\textbf{24}                                       & \textit{0.868} 
 & 118800
 
\\ \hline 
\textbf{26}                                       & \textit{0.868} 
 & 170044

\\ \hline
\end{tabular}
\caption{MCC of the trained SVM hyperplanes with varying degree polynomial kernels, reported to 3 decimal places; shown with the hyperplane equation number of terms.}
\label{tab:SVMRAW}
\end{table}
}

\paragraph{\textit{Higher Edge Multiplicities}}\mbox{}\\
As a final remark on these results, we note again that the analytic expression for the separating hyperplanes was derived from a dataset that contained rank 4 quivers with either $0$, $1$, or $2$ edges between each vertex. 
An immediate question to ask is how well the hyperplane generalises to classify rank 4 quivers with edge multiplicities $\geq 3$. 

In testing this, we trial the machine learnt degree 18 polynomial kernel equation on a selection of quivers with higher edge multiplicities; as shown in §\ref{app:hem_quivers}, all of which are classified correctly.
These quivers include the A4-like type (Figure \ref{3A4}) and NMA type (Figure \ref{3NM}) edited to include edges with multiplicity $\geq 3$, as well as a quiver for each of these reached after a random sequence of mutations (Figures \ref{3A4m}, \ref{3NMm}), finally each of these types is also shown with much larger edge multiplicities (Figures \ref{3A4p}, \ref{3NMp}). 

On this paper's respective \href{https://github.com/KTKAW/MACHINE_LEARNING_MUTATION_ACYCLICITY_OF_QUIVERS.git}{\texttt{GitHub}} repository a script is also available to test any quiver of the reader's choosing for mutation-acyclicity using this trained degree-18 SVM; simply enter the adjacency matrix of the quiver and the script will output a prediction.
How the performance extends in general beyond Dataset 4, to higher edge multiplicities, and how the mutation invariant defined by the trained SVM polynomial can be generalised in these cases is left to future work.

\section{Conclusion}\label{sec:conc}

In this paper, we explored the potential for ML techniques to tackle a complex problem within quiver mutation, specifically in detecting mutation-acyclicity. Our experiments yielded strong results, with NNs and SVMs achieving high accuracy in distinguishing mutation-acyclic quivers from their non-mutation-acyclic counterparts. 

While these findings affirm the efficacy of ML in this domain, they also open new avenues for further investigation.
The explicit equation formulations learnt by the SVMs of each kernel degree (available on \href{https://github.com/KTKAW/MACHINE_LEARNING_MUTATION_ACYCLICITY_OF_QUIVERS.git}{\texttt{GitHub}}) aim to provide a starting basis for distilling mathematical insight about the design of mutation-invariant combinatorial constants that identify mutation-acyclicity, alike the Markov constant, for higher rank quivers. 
This interpretative phase will be crucial for translating the empirical success of ML into theoretical advancements. 

Additionally, the successes of these ML investigations motivated the computational work in proving the new Theorem \ref{ca_theorem}, another central result of this paper. 
It is worth emphasising in the previously classified rank 3 case of Theorem \ref{theorem:subquiver} (2), the exceptional cases for identifying mutation-acyclicity all have edge weights $\leq 2$, as has been considered for this Theorem in the rank 4 case. 
The new Theorem's low edge-weight results for rank 4, hence hope to be useful for future work for a more general result by providing an exhaustive check of these likely exceptional cases, as well as to be useful in practical work where academics often use low edge-weight quivers in cluster algebra seeding. 

In summary, this work not only establishes a solid foundation for the use of ML in quiver mutation analysis but also provides additional evidence for the exciting possibilities that lie ahead in integrating data-driven approaches with traditional mathematical research. By continuing to embrace this ML framework, we aim to deepen our understanding of mathematical structures and contribute to the broader intersection of machine learning and mathematics research.
\section*{Acknowledgements}
KAW is supported by a studentship from the UK Engineering and Physical Sciences Research Council (EPSRC) and by the UK Science and Technology Facilities Council (STFC) Consolidated Grant ST/P000754/1
“String theory, gauge theory and duality”. 
EH acknowledges support from Pierre Andurand over the course of this research.
KHL's work was partially supported by a grant from the Simons Foundation (\#712100).
This research utilised Queen Mary's Apocrita HPC facility \cite{apocrita}, supported by QMUL Research-IT.



\appendix
\section{Appendix}
\subsection{Machine Learning Metrics}\label{app:ml_intro}\mbox{}\\
Test performances in supervised classification problems are recorded in a confusion matrix: $C_{ij}$.
In the binary classification scenario, this is a $2 \times 2$ matrix where the row index is the true class of the input and the column index is the model-predicted class of the input.
As the test data is run through the model, the respective $C_{ij}$ entry is incremented according to each test input's true and predicted class; then, after testing is complete, the matrix is appropriately normalised.
All popular classification performance measures are then computed as functions of this matrix, the most common being accuracy, which is just the proportion of correctly classified inputs and is equal to the matrix trace.
However, this measure is prone to bias where the true classes are unbalanced; therefore, a more appropriate measure, which includes contribution from the off-diagonal components (the statistical Type I and Type II errors), is Matthew's correlation coefficient (MCC), defined, along with accuracy, as:
\begin{align}
    \text{Accuracy } & \vcentcolon = \sum_i C_{ii} = C_{00} + C_{11}\;,\\
    \text{MCC } & \vcentcolon = \frac{C_{11}C_{00}-C_{01}C_{10}}{\sqrt{(C_{11}+C_{01})(C_{11}+C_{10})(C_{00}+C_{01})(C_{00}+C_{10})}}\;,
\end{align}
both measures evaluate to 1 for perfect learning, whereas for no learning, accuracy evaluates to 0.5 (for balanced binary classes) and MCC to 0. 

In these classification problems, the output is often a discrete probability distribution over the possible classes, from which the loss function used in training, binary cross-entropy, is defined as:
\begin{equation}
    \text{Binary Cross-Entropy} \vcentcolon = \frac{1}{N}\sum_i^N (p_{true=0,i}\text{log}(p_{pred=0,i}) + p_{true=1,i}\text{log}(p_{pred=1,i}))\;,
\end{equation}
where the average is taken over the training batch of size $N$, using the true probabilities $p_{true=\#,i}$ of the output being in either class 0 or 1 (which as the data is labelled is exactly 1 for one probability and 0 for the other), and the predicted probabilities $p_{pred=\#,i}$ (satisfying $1=p_{pred=0,i}+p_{pred=1,i} \ \forall i$).

\subsubsection{\textit{Neural Network Activations}}\label{app:ml_intros_acts}\mbox{}\\
Fundamental to the structure of neural network architectures is their non-linearity, introduced by the choice of an activation function.
The standard options, as used in the architectures in this work, are detailed below.

The choice of activation function determines the structure of the function approximation; the most common choice is the rectified linear unit (ReLU), defined:
\begin{equation}
    \phi_{\text{ReLU}}(x) \vcentcolon =  \begin{cases} x \quad & x\geq 0, \\ 0 \quad & \text{otherwise}. \end{cases}\;
\end{equation}

This choice sets the full NN function to be piecewise linear, and then it is this function style used to approximate the true map connecting the input to the output spaces.
There are generalisations of ReLU which smoothen the piecewise linear nature, a candidate used here is the scaled exponential linear unit (SELU), defined:
\begin{equation}
    \phi_{\text{SELU}}(x) \vcentcolon =  \begin{cases} \lambda x \quad & x\geq 0, \\ \lambda \alpha (e^x -1) \quad & \text{otherwise} , \end{cases}\;
\end{equation}
for constants $(\lambda, \alpha)$ taking their standard values $(1.05070098,1.67326324)$. 
Another common activation also used is the standard hyperbolic tangent function, $\phi_{\text{tanh}}(x) = \text{tanh}(x)$.
For the final output layer, where output normalisation to a probability distribution is desired, the softmax activation function is used:
\begin{equation}
    \phi_{\text{softmax}}(x_i) \vcentcolon = \frac{e^{x_i}}{\sum_j e^{x_j}}\;,
\end{equation}
acting on the $i$th node output in the output layer $x_i$, normalising using a sum over all $j$ nodes in the output layer.

\subsection{Initial Quivers}\label{app:quivers}\mbox{}\\
Here, the initial quivers used to seed the generation of quivers in the mutation classes considered are shown.
The generated mutation classes of quivers were then combined in various ways to form the datasets, described in §\ref{sec:data}, used in the ML investigations of §\ref{sec:ml}.

\begin{figure}[H]
    \centering
    \begin{subfigure}[b]{0.3\textwidth}
         \centering
         \begin{tikzcd}[row sep={4em,between origins}, column sep={4em,between origins}]
            1 & 2 & 3 & 4 
            \arrow[from=1-1, to=1-2, shorten=0.2em]
            \arrow[from=1-2, to=1-3, shorten=0.2em]
            \arrow[from=1-3, to=1-4, shorten=0.2em]
         \end{tikzcd}
         \caption{A4 Type 1}
         \label{A41}
    \end{subfigure} \hfill
    \begin{subfigure}[b]{0.3\textwidth}
         \centering
         \begin{tikzcd}[row sep={4em,between origins}, column sep={4em,between origins}]
            1 & 2 & 3 & 4 
            \arrow[from=1-1, to=1-2, "2", shorten=0.2em]
            \arrow[from=1-2, to=1-3, "2", shorten=0.2em]
            \arrow[from=1-3, to=1-4, "2", shorten=0.2em]
         \end{tikzcd}
         \caption{A4 Type 2}
         \label{A42}
    \end{subfigure} \hfill
    \begin{subfigure}[b]{0.3\textwidth}
         \centering
         \begin{tikzcd}[row sep={4em,between origins}, column sep={4em,between origins}]
            1 & 2 & 3 & 4 
            \arrow[from=1-1, to=1-2, shorten=0.2em]
            \arrow[from=1-3, to=1-2, "2"above, shorten=0.2em]
            \arrow[from=1-3, to=1-4, shorten=0.2em]
         \end{tikzcd}
         \caption{A4 Type 3}
         \label{A43}
    \end{subfigure}
    \caption{A4-like Quivers}
    \label{A4_GRAPHS}
\end{figure}

\begin{figure}[H]
     \centering
     \begin{subfigure}[b]{0.3\textwidth}
         \centering
         \begin{tikzcd}[row sep={4em,between origins}, column sep={4em,between origins}]
            2 & 1 & 3 \\
            & 4 &
            \arrow[from=1-2, to=1-1, shorten=0.2em]
            \arrow[from=1-2, to=1-3, shorten=0.2em]
            \arrow[from=1-2, to=2-2, shorten=0.2em]
         \end{tikzcd}
         \caption{D4 Type 1}
         \label{D41}
     \end{subfigure}
     \quad
     \begin{subfigure}[b]{0.3\textwidth}
         \centering
         \begin{tikzcd}[row sep={4em,between origins}, column sep={4em,between origins}]
            2 & 1 & 3 \\
            & 4 &
            \arrow[from=1-2, to=1-1, shorten=0.2em]
            \arrow[from=1-2, to=1-3, shorten=0.2em]
            \arrow[from=1-2, to=2-2, "2", shorten=0.2em]
         \end{tikzcd}
         \caption{D4 Type 2}
         \label{D42}
     \end{subfigure}
        \caption{D4-like Quivers}
        \label{D4_GRAPHS}
\end{figure}

\begin{figure}[H]
     \centering
     \begin{subfigure}[b]{0.3\textwidth}
         \centering
         \begin{tikzcd}[row sep={4em,between origins}, column sep={4em,between origins}]
            1 & 2 & 4 \\
            & 3 &
            \arrow[from=1-1, to=1-2, "2", shorten=0.2em]
            \arrow[from=1-2, to=2-2, "2", shorten=0.2em]
            \arrow[from=2-2, to=1-1, "2", shorten=0.2em]
            \arrow[from=1-2, to=1-3, "2", shorten=0.2em]
         \end{tikzcd}
         \caption{NMA Type 1}
         \label{NMA1}
     \end{subfigure}
     \quad
     \begin{subfigure}[b]{0.3\textwidth}
         \centering
         \begin{tikzcd}[row sep={4em,between origins}, column sep={4em,between origins}]
            1 & 2 \\
            4 & 3 
            \arrow[from=1-1, to=1-2, "2", shorten=0.2em]
            \arrow[from=1-2, to=2-2, "2", shorten=0.2em]
            \arrow[from=2-2, to=2-1, "2", shorten=0.2em]
            \arrow[from=2-1, to=1-1, "2", shorten=0.2em]
         \end{tikzcd}
         \caption{NMA Type 2}
         \label{NMA2}
     \end{subfigure}
        \caption{NMA Quivers}
        \label{NMA_GRAPHS}
\end{figure} 

\begin{figure}[H]
     \centering
     \begin{subfigure}[b]{0.3\textwidth}
         \centering
         \begin{tikzcd}[row sep={4em,between origins}, column sep={4em,between origins}]
            1 & 2 \\
            4 & 3 
            \arrow[from=1-1, to=1-2, "3", shorten=0.2em]
            \arrow[from=1-2, to=2-2, "2", shorten=0.2em]
            \arrow[from=2-2, to=2-1, "3", shorten=0.2em]
            \arrow[from=2-1, to=1-1, "2", shorten=0.2em]
         \end{tikzcd}
         \caption{Box Quiver 1}
         \label{B2}
     \end{subfigure}
     \quad
     \begin{subfigure}[b]{0.3\textwidth}
         \centering
         \begin{tikzcd}[row sep={4em,between origins}, column sep={4em,between origins}]
            1 & 2 \\
            4 & 3 
            \arrow[from=1-1, to=1-2, "3", shorten=0.2em]
            \arrow[from=1-2, to=2-2, "3", shorten=0.2em]
            \arrow[from=2-2, to=2-1, "3", shorten=0.2em]
            \arrow[from=2-1, to=1-1, "3", shorten=0.2em]
         \end{tikzcd}
         \caption{Box Quiver 2}
         \label{B3}
     \end{subfigure}
        \caption{NMA Quivers Box Type}
        \label{BOX_GRAPHS}
\end{figure}

\begin{figure}[H]
     \centering
     \begin{subfigure}[b]{0.3\textwidth}
         \centering
         \begin{tikzcd}[row sep={4em,between origins}, column sep={4em,between origins}]
            & 4 \\
            & 2 \\
            1 & & 3
            \arrow[from=3-1, to=3-3, "2", shorten=0.2em]
            \arrow[from=3-3, to=2-2, "2", shorten=0.2em]
            \arrow[from=2-2, to=3-1, "2", shorten=0.2em]
            \arrow[from=3-1, to=1-2, shorten=0.2em]
            \arrow[from=3-3, to=1-2, shorten=0.2em]
            \arrow[from=2-2, to=1-2, shorten=0.2em]
         \end{tikzcd}
         \caption{M1}
         \label{M1}
     \end{subfigure}
     \hfill
     \begin{subfigure}[b]{0.3\textwidth}
         \centering
         \begin{tikzcd}[row sep={4em,between origins}, column sep={4em,between origins}]
            & 4 \\
            & 2 \\
            1 & & 3 
            \arrow[from=3-1, to=3-3, "2", shorten=0.2em]
            \arrow[from=3-3, to=2-2, "2", shorten=0.2em]
            \arrow[from=2-2, to=3-1, "2", shorten=0.2em]
            \arrow[from=1-2, to=3-1, shorten=0.2em]
            \arrow[from=3-3, to=1-2, shorten=0.2em]
            \arrow[from=2-2, to=1-2, shorten=0.2em]
         \end{tikzcd}
         \caption{M2}
         \label{M2}
     \end{subfigure}
     \hfill
     \begin{subfigure}[b]{0.3\textwidth}
         \centering
         \begin{tikzcd}[row sep={4em,between origins}, column sep={4em,between origins}]
            & 4 \\
            & 2 \\
            1 & & 3 
            \arrow[from=3-1, to=3-3, "2", shorten=0.2em]
            \arrow[from=3-3, to=2-2, "2", shorten=0.2em]
            \arrow[from=2-2, to=3-1, "2", shorten=0.2em]
            \arrow[from=1-2, to=3-1, shorten=0.2em]
            \arrow[from=1-2, to=3-3, shorten=0.2em]
            \arrow[from=2-2, to=1-2, shorten=0.2em]
         \end{tikzcd}
         \caption{M3}
         \label{M3}
     \end{subfigure}
        \caption{NMA-like Graphs Set 1}
        \label{NMA_LIKE_S1_GRAPHS}
\end{figure}

\begin{figure}[H]
     \centering
     \begin{subfigure}[b]{0.3\textwidth}
         \centering
         \begin{tikzcd}[row sep={4em,between origins}, column sep={4em,between origins}]
            & 4 \\
            & 2 \\
            1 & & 3
            \arrow[from=3-1, to=3-3, "3", shorten=0.2em]
            \arrow[from=3-3, to=2-2, "3", shorten=0.2em]
            \arrow[from=2-2, to=3-1, "3", shorten=0.2em]
            \arrow[from=3-1, to=1-2, shorten=0.2em]
            \arrow[from=3-3, to=1-2, shorten=0.2em]
            \arrow[from=2-2, to=1-2, shorten=0.2em]
         \end{tikzcd}
         \caption{NM1}
         \label{S21}
     \end{subfigure}
     \hfill
     \begin{subfigure}[b]{0.3\textwidth}
         \centering
         \begin{tikzcd}[row sep={4em,between origins}, column sep={4em,between origins}]
            & 4 \\
            & 2 \\
            1 & & 3 
            \arrow[from=3-1, to=3-3, "3", shorten=0.2em]
            \arrow[from=3-3, to=2-2, "3", shorten=0.2em]
            \arrow[from=2-2, to=3-1, "3", shorten=0.2em]
            \arrow[from=1-2, to=3-1, shorten=0.2em]
            \arrow[from=3-3, to=1-2, shorten=0.2em]
            \arrow[from=2-2, to=1-2, shorten=0.2em]
         \end{tikzcd}
         \caption{NM2}
         \label{S22}
     \end{subfigure}
     \hfill
     \begin{subfigure}[b]{0.3\textwidth}
         \centering
         \begin{tikzcd}[row sep={4em,between origins}, column sep={4em,between origins}]
            & 4 \\
            & 2 \\
            1 & & 3 
            \arrow[from=3-1, to=3-3, "3", shorten=0.2em]
            \arrow[from=3-3, to=2-2, "3", shorten=0.2em]
            \arrow[from=2-2, to=3-1, "3", shorten=0.2em]
            \arrow[from=1-2, to=3-1, shorten=0.2em]
            \arrow[from=1-2, to=3-3, shorten=0.2em]
            \arrow[from=2-2, to=1-2, shorten=0.2em]
         \end{tikzcd}
         \caption{NM3}
         \label{S23}
     \end{subfigure}
        \caption{NMA-like Graphs Set 2}
        \label{NMA_LIKE_S2_GRAPHS}
\end{figure}  

\subsubsection{\textit{High Edge Multiplicity Quivers}}\label{app:hem_quivers}\mbox{}\\
A selection of quivers with higher-multiplicity edges is shown below, all of which are correctly classified by the trained degree-18 SVM polynomial.

\begin{figure}[H]
     \centering
     \begin{subfigure}[b]{0.3\textwidth}
         \centering
         \begin{tikzcd}[row sep={4em,between origins}, column sep={4em,between origins}]
            1 & 2 & 3 & 4
            \arrow[from=1-1, to=1-2, "3", shorten=0.2em]
            \arrow[from=1-3, to=1-2, "4", shorten=0.2em]
            \arrow[from=1-3, to=1-4, "3", shorten=0.2em]
         \end{tikzcd}
         \caption{A4 Type \\ Multiplicity $\geq3$}
         \label{3A4}
     \end{subfigure}
     \hfill
     \begin{subfigure}[b]{0.3\textwidth}
         \centering
         \begin{tikzcd}[row sep={4em,between origins}, column sep={4em,between origins}]
            & 4 \\
            & 2 \\
            1 & & 3 
            \arrow[from=3-1, to=3-3, "96", shorten=0.2em]
            \arrow[from=2-2, to=3-3, "32"left, shorten=0.2em]
            \arrow[from= 2-2 , to=3-1, "381"below, shorten=0.2em]
            \arrow[from=1-2, to=3-1, "252"left, shorten=0.2em]
            \arrow[from=3-3, to=1-2, "3"right, shorten=0.2em]
            \arrow[from=1-2, to=2-2, "96", shorten=0.2em]
         \end{tikzcd}
         \caption{A4 Type Multiplicity $\geq3$,\\ \phantom{MMMMMM} Mutated}
         \label{3A4m}
     \end{subfigure}
     \hfill
    \begin{subfigure}[b]{0.3\textwidth}
         \centering
         \begin{tikzcd}[row sep={4em,between origins}, column sep={4em,between origins}]
            & 4 \\
            & 2 \\
            1 & & 3 
            \arrow[from=3-1, to=3-3, "9.4\times10^{15}", shorten=0.2em]
            \arrow[from=3-3, to=2-2, "8.3\times10^{5}"left, shorten=0.2em]
            \arrow[from=2-2, to=3-1, "1.1 \times10^{10}"below right, shorten=0.2em]
            \arrow[from=3-1, to=1-2, "8.1\times10^{15}"left, shorten=0.2em]
            \arrow[from=3-3, to=1-2, "2.1\times10^{15}"right, shorten=0.2em]
            \arrow[from=2-2, to= 1-2, "2.5\times10^{9}", shorten=0.2em]
         \end{tikzcd}
         \caption{A4 Type Multiplicity $\gg 3$ \\ \phantom{MMMMMM} Mutated}
         \label{3A4p}
    \end{subfigure}\\ \vspace{5mm}
    \begin{subfigure}[b]{0.3\textwidth}
         \centering
         \begin{tikzcd}[row sep={4em,between origins}, column sep={4em,between origins}]
            & 4 \\
            & 3 \\
            2 & & 1 
            \arrow[from=3-1, to=3-3, "3", shorten=0.2em]
            \arrow[from=3-3, to=2-2, "6 ", shorten=0.2em]
            \arrow[from=2-2, to=3-1, "3", shorten=0.2em]
            \arrow[from=2-2, to=1-2, shorten=0.2em]
         \end{tikzcd}
         \caption{NMA Type \\ Multiplicity $\geq3$}
         \label{3NM}
     \end{subfigure}
     \hfill
     \begin{subfigure}[b]{0.3\textwidth}
         \centering
         \begin{tikzcd}[row sep={4em,between origins}, column sep={4em,between origins}]
            & 4 \\
            & 3 \\
            1 & & 2 
            \arrow[from=3-1, to=3-3, "3", shorten=0.2em]
            \arrow[from=3-3, to=2-2, "3", shorten=0.2em]
            \arrow[from=2-2, to=3-1, "3", shorten=0.2em]
            \arrow[from=2-2, to=1-2, shorten=0.2em]
         \end{tikzcd}
         \caption{NMA Type Multiplicity $\geq3$,\\ \phantom{MMMMMM} Mutated}
         \label{3NMm}
     \end{subfigure}
     \hfill
    \begin{subfigure}[b]{0.3\textwidth}
         \centering
         \begin{tikzcd}[row sep={4em,between origins}, column sep={4em,between origins}]
            1 & 2 \\
            4 & 3
            \arrow[from=1-1, to=1-2, "61", shorten=0.2em]
            \arrow[from=1-2, to=2-2, "223", shorten=0.2em]
            \arrow[from=2-2, to=2-1, "61", shorten=0.2em]
            \arrow[from=2-1, to=1-1, "223", shorten=0.2em]
         \end{tikzcd}
         \caption{NMA Type \\ Multiplicity $\gg 3$}
         \label{3NMp}
    \end{subfigure}
    \caption{}
    \label{highmultiplicityquivers}    
\end{figure} 

\subsection{NN Architecture Diagrammatics}\label{app:nndiagrams}\mbox{}\\
Here are flowcharts displaying the NN architecture setups used in the learning of §\ref{sec:nn}.
Figure \ref{First_Experiment_Class_NN} is the architecture used for the experiments for §\ref{sec:exp1} and §\ref{sec:exp2}; whereas Figure \ref{FIRST_EXP_NN_DIAGRAM_BIG_DATA} is for §\ref{sec:exp3}.

\begin{figure}[H]
    \centering
    \includegraphics[width=0.55\textwidth]{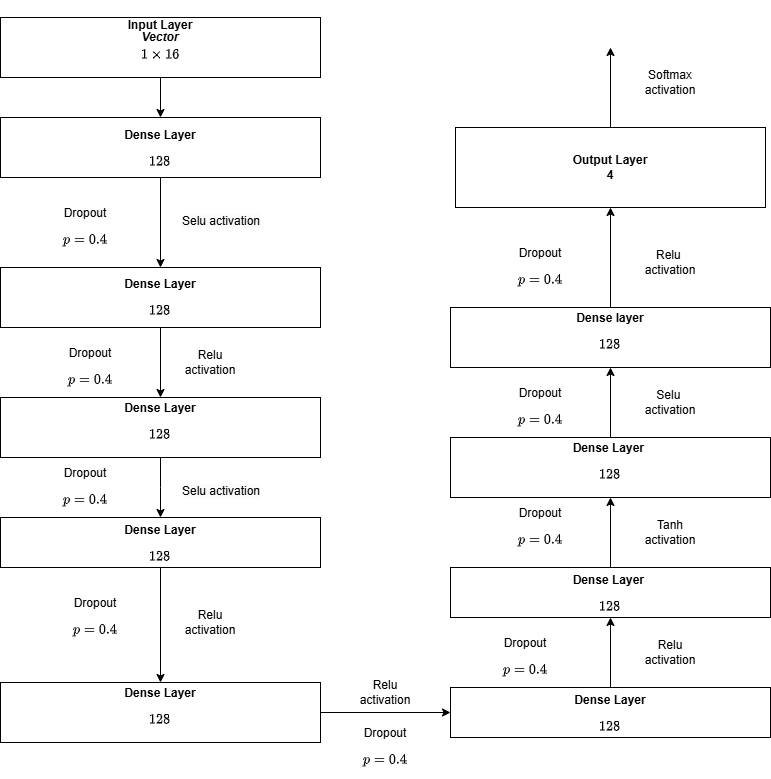}
    \caption{Diagram of the NN architecture used in experiments §\ref{sec:exp1} and §\ref{sec:exp2}.}
    \label{First_Experiment_Class_NN}
\end{figure}
\begin{figure}[H]
    \centering
    \includegraphics[width=0.25\textwidth]{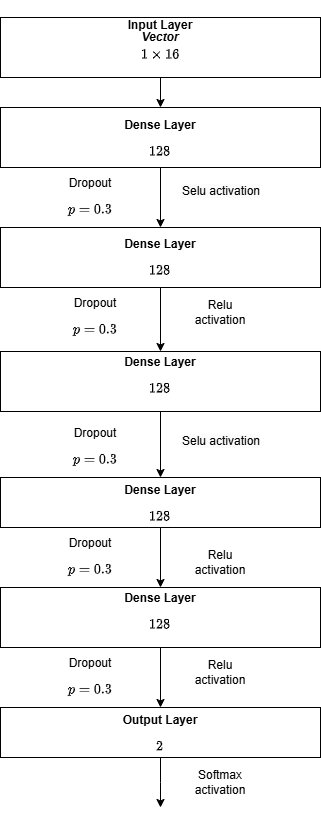}
    \caption{Diagram of the NN architecture used in experiment §\ref{sec:exp3}.}
    \label{FIRST_EXP_NN_DIAGRAM_BIG_DATA}
\end{figure}

\newpage                         
{\bf Declarations}

{\bf Conflicts of interests/Competing interests} {The authors have no conflicts of interests or competing interests
to declare.}

{\bf Data Availability} The datasets generated during the current study are shared at \href{https://github.com/KTKAW/MACHINE_LEARNING_MUTATION_ACYCLICITY_OF_QUIVERS.git}{\texttt{GitHub}}\footnote{\href{https://github.com/KTKAW/MACHINE_LEARNING_MUTATION_ACYCLICITY_OF_QUIVERS.git}{\texttt{https://github.com/KTKAW/MACHINE_LEARNING_MUTATION_ACYCLICITY_OF_QUIVERS.git}}}.

\bibliographystyle{plain}
\bibliography{references}

\begin{thebibliography}{10}

\bibitem{tensorflow2015whitepaper}
Mart\'{i}n Abadi, Ashish Agarwal, Paul Barham, Eugene Brevdo, Zhifeng Chen, Craig Citro, Greg~S. Corrado, Andy Davis, Jeffrey Dean, Matthieu Devin, Sanjay Ghemawat, Ian Goodfellow, Andrew Harp, Geoffrey Irving, Michael Isard, Yangqing Jia, Rafal Jozefowicz, Lukasz Kaiser, Manjunath Kudlur, Josh Levenberg, Dandelion Man\'{e}, Rajat Monga, Sherry Moore, Derek Murray, Chris Olah, Mike Schuster, Jonathon Shlens, Benoit Steiner, Ilya Sutskever, Kunal Talwar, Paul Tucker, Vincent Vanhoucke, Vijay Vasudevan, Fernanda Vi\'{e}gas, Oriol Vinyals, Pete Warden, Martin Wattenberg, Martin Wicke, Yuan Yu, and Xiaoqiang Zheng.
\newblock {TensorFlow}: Large-scale machine learning on heterogeneous systems, 2015.
\newblock Software available from tensorflow.org.

\bibitem{Aggarwal:2023swe}
Daattavya Aggarwal, Yang-Hui He, Elli Heyes, Edward Hirst, Henrique N.~S\'a Earp, and Tom\'as S.~R. Silva.
\newblock {Machine learning Sasakian and G2 topology on contact Calabi-Yau 7-manifolds}.
\newblock {\em Phys. Lett. B}, 850:138517, 2024.

\bibitem{Alawadhi:2023gxa}
Rashid Alawadhi, Daniele Angella, Andrea Leonardo, and Tancredi~Schettini Gherardini.
\newblock {Constructing and Machine Learning Calabi-Yau Five-Folds}.
\newblock {\em Fortsch. Phys.}, 72(2):2300262, 2024.

\bibitem{Amir:2022sab}
Malik Amir, Yang-Hui He, Kyu-Hwan Lee, Thomas Oliver, and Eldar Sultanow.
\newblock {Machine Learning Class Numbers of Real Quadratic Fields}.
\newblock 9 2022.

\bibitem{Arias-Tamargo:2022qgb}
Guillermo Arias-Tamargo, Yang-Hui He, Elli Heyes, Edward Hirst, and Diego Rodriguez-Gomez.
\newblock {Brain webs for brane webs}.
\newblock {\em Phys. Lett. B}, 833:137376, 2022.

\bibitem{Ashmore:2023ajy}
Anthony Ashmore, Yang-Hui He, Elli Heyes, and Burt~A. Ovrut.
\newblock {Numerical spectra of the Laplacian for line bundles on Calabi-Yau hypersurfaces}.
\newblock {\em JHEP}, 07:164, 2023.

\bibitem{assem2008mutation}
Ibrahim Assem, Martin Blais, Thomas Br{\"u}stle, and Audrey Samson.
\newblock Mutation classes of skew-symmetric 3$\times$ 3-matrices.
\newblock {\em Communications in Algebra{\textregistered}}, 36(4):1209--1220, 2008.

\bibitem{Assem2006Elements}
Ibrahim Assem, Daniel Simson, and Andrzej Skowro\'{n}ski.
\newblock {\em Elements of the representation theory of associative algebras. {V}ol. 1}, volume~65 of {\em London Mathematical Society Student Texts}.
\newblock Cambridge University Press, Cambridge, 2006.
\newblock Techniques of representation theory.

\bibitem{Bao:2020nbi}
Jiakang Bao, Sebasti\'an Franco, Yang-Hui He, Edward Hirst, Gregg Musiker, and Yan Xiao.
\newblock {Quiver Mutations, Seiberg Duality and Machine Learning}.
\newblock {\em Phys. Rev. D}, 102(8):086013, 2020.

\bibitem{Bao:2021olg}
Jiakang Bao, Yang-Hui He, and Edward Hirst.
\newblock {Neurons on amoebae}.
\newblock {\em J. Symb. Comput.}, 116:1--38, 2023.

\bibitem{BBH2011Cluster}
Andre Beineke, Thomas Br\"{u}stle, and Lutz Hille.
\newblock Cluster-cyclic quivers with three vertices and the {M}arkov equation.
\newblock {\em Algebr. Represent. Theory}, 14(1):97--112, 2011.
\newblock With an appendix by Otto Kerner.

\bibitem{Berglund:2023ztk}
Per Berglund, Yang-Hui He, Elli Heyes, Edward Hirst, Vishnu Jejjala, and Andre Lukas.
\newblock {New Calabi\textendash{}Yau manifolds from genetic algorithms}.
\newblock {\em Phys. Lett. B}, 850:138504, 2024.

\bibitem{Berman:2021mcw}
David~S. Berman, Yang-Hui He, and Edward Hirst.
\newblock {Machine learning Calabi-Yau hypersurfaces}.
\newblock {\em Phys. Rev. D}, 105(6):066002, 2022.

\bibitem{Bongartz1983Algebras}
Klaus Bongartz.
\newblock Algebras and quadratic forms.
\newblock {\em J. London Math. Soc. (2)}, 28(3):461--469, 1983.

\bibitem{BMR2008Cluster}
Aslak~Bakke Buan, Bethany~R. Marsh, and Idun Reiten.
\newblock Cluster mutation via quiver representations.
\newblock {\em Comment. Math. Helv.}, 83(1):143--177, 2008.

\bibitem{Capuozzo:2024vdw}
Pietro Capuozzo, Tancredi~Schettini Gherardini, and Benjamin Suzzoni.
\newblock {Machine Learning Toric Duality in Brane Tilings}.
\newblock 9 2024.

\bibitem{Chen:2022jwd}
Siqi Chen, Yang-Hui He, Edward Hirst, Andrew Nestor, and Ali Zahabi.
\newblock {Mahler measuring the genetic code of amoebae}.
\newblock {\em Adv. Theor. Math. Phys.}, 27(5):1405--1461, 2023.

\bibitem{Cheung:2022itk}
Man-Wai Cheung, Pierre-Philippe Dechant, Yang-Hui He, Elli Heyes, Edward Hirst, and Jian-Rong Li.
\newblock {Clustering Cluster Algebras with Clusters}.
\newblock 12 2022.

\bibitem{quanta}
Lyndie Chiou.
\newblock Elliptic curve ‘murmurations’ found with {AI} take flight.
\newblock {\em Qunata Magazine}, 2024.

\bibitem{Dechant:2022ccf}
Pierre-Philippe Dechant, Yang-Hui He, Elli Heyes, and Edward Hirst.
\newblock {Cluster Algebras: Network Science and Machine Learning}.
\newblock {\em J. Comput. Algebra}, 8, 2023.

\bibitem{EJLN2024Geometry}
Tucker~J. Ervin, Blake Jackson, Kyungyong Lee, and Son~Dang Nguyen.
\newblock Geometry of $c$-matrices for mutation-infinite quivers.
\newblock In {\em {F}ormal {P}ower {S}eries and {A}lgebraic {C}ombinatorics 2024}, volume~93 of {\em S\'{e}m. Lothar. Combin.}, page to appear. 2024.

\bibitem{Felikson2018Acyclic}
Anna Felikson and Pavel Tumarkin.
\newblock Acyclic cluster algebras, reflection groups, and curves on a punctured disc.
\newblock {\em Adv. Math.}, 340:855--882, 2018.

\bibitem{Feng_2001}
Bo~Feng, Amihay Hanany, Yang-Hui He, and Angel~M Uranga.
\newblock Toric duality as {Seiberg} duality and brane diamonds.
\newblock {\em Journal of High Energy Physics}, 2001(12):035–035, Dec 2001.

\bibitem{fomin_long_2023}
Sergey Fomin and Scott Neville.
\newblock Long mutation cycles, April 2023.
\newblock arXiv:2304.11505 [math].

\bibitem{Fomin2002clusterI}
Sergey Fomin and Andrei Zelevinsky.
\newblock Cluster algebras. {I}. {F}oundations.
\newblock {\em J. Amer. Math. Soc.}, 15(2):497--529, 2002.

\bibitem{GHKK2018Canonical}
Mark Gross, Paul Hacking, Sean Keel, and Maxim Kontsevich.
\newblock Canonical bases for cluster algebras.
\newblock {\em J. Amer. Math. Soc.}, 31(2):497--608, 2018.

\bibitem{Hagberg2008ExploringNS}
Aric~A. Hagberg, Daniel~A. Schult, and Pieter Swart.
\newblock Exploring network structure, dynamics, and function using networkx.
\newblock 2008.

\bibitem{Hashemi:2024azx}
Baran Hashemi, Roderic~G. Corominas, and Alessandro Giacchetto.
\newblock {Can Transformers Do Enumerative Geometry?}
\newblock 8 2024.

\bibitem{He_2022_n}
Yang-Hui He, Kyu-Hwan Lee, and Thomas Oliver.
\newblock Machine-learning number fields.
\newblock {\em Computation and Geometry of Data}, 2:49--66, 2022.

\bibitem{He_2022}
Yang-Hui He, Kyu-Hwan Lee, and Thomas Oliver.
\newblock Machine-learning the {S}ato–{T}ate conjecture.
\newblock {\em J. Symb. Comput.}, 111:61–72, 2022.

\bibitem{He_2023}
Yang-Hui He, Kyu-Hwan Lee, and Thomas Oliver.
\newblock Machine-learning invariants of arithmetic curve.
\newblock {\em J. Symb. Comput.}, 115:478--491, 2023.

\bibitem{He:2022pqn}
Yang-Hui He, Kyu-Hwan Lee, Thomas Oliver, and Alexey Pozdnyakov.
\newblock {Murmurations of elliptic curves}.
\newblock {\em Exp. Math.}, 2024.

\bibitem{Hirst:2023kdl}
Edward Hirst and Tancredi~Schettini Gherardini.
\newblock {Calabi-Yau four-, five-, sixfolds as Pwn hypersurfaces: Machine learning, approximation, and generation}.
\newblock {\em Phys. Rev. D}, 109(10):106006, 2024.

\bibitem{Kac1980Infinite}
V.~G. Kac.
\newblock Infinite root systems, representations of graphs and invariant theory.
\newblock {\em Invent. Math.}, 56(1):57--92, 1980.

\bibitem{Keller2017Quiver}
Bernhard Keller.
\newblock Quiver mutation and combinatorial {DT-invariants}.
\newblock September 2017.

\bibitem{apocrita}
Thomas King, Simon Butcher, and Lukasz Zalewski.
\newblock {\em {Apocrita - High Performance Computing Cluster for Queen Mary University of London}}, March 2017.

\bibitem{kingma2017adam}
Diederik~P. Kingma and Jimmy Ba.
\newblock Adam: A method for stochastic optimization, 2017.

\bibitem{Kontsevich2008Stability}
Maxim Kontsevich and Yan Soibelman.
\newblock Stability structures, motivic {Donaldson-Thomas} invariants and cluster transformations.
\newblock November 2008.

\bibitem{LL2021Correspondence}
Kyu-Hwan Lee and Kyungyong Lee.
\newblock A correspondence between rigid modules over path algebras and simple curves on {R}iemann surfaces.
\newblock {\em Exp. Math.}, 30(3):315--331, 2021.

\bibitem{LLM2023Geometric}
Kyu-Hwan Lee, Kyungyong Lee, and Matthew~R. Mills.
\newblock Geometric description of {$C$}-vectors and real {L}\"{o}sungen.
\newblock {\em Math. Z.}, 303(2):Paper No. 44, 33, 2023.

\bibitem{Muller2013Locally}
Greg Muller.
\newblock Locally acyclic cluster algebras.
\newblock {\em Adv. Math.}, 233:207--247, 2013.

\bibitem{Muller2014AU}
Greg Muller.
\newblock {$A=U$} for locally acyclic cluster algebras.
\newblock {\em SIGMA Symmetry Integrability Geom. Methods Appl.}, 10:Paper 094, 8, 2014.

\bibitem{musiker2011compendium}
Gregg Musiker and Christian Stump.
\newblock A compendium on the cluster algebra and quiver package in sage, 2011.

\bibitem{Chavez2015Acyclic}
Alfredo N\'{a}jera~Ch\'{a}vez.
\newblock On the c-vectors of an acyclic cluster algebra.
\newblock {\em Int. Math. Res. Not. IMRN}, (6):1590--1600, 2015.

\bibitem{sklearn}
F.~Pedregosa, G.~Varoquaux, A.~Gramfort, V.~Michel, B.~Thirion, O.~Grisel, M.~Blondel, P.~Prettenhofer, R.~Weiss, V.~Dubourg, J.~Vanderplas, A.~Passos, D.~Cournapeau, M.~Brucher, M.~Perrot, and E.~Duchesnay.
\newblock Scikit-learn: Machine learning in {P}ython.
\newblock {\em Journal of Machine Learning Research}, 12:2825--2830, 2011.

\bibitem{peixoto_graph-tool_2014}
Tiago~P. Peixoto.
\newblock The graph-tool python library.
\newblock {\em figshare}, 2014.

\bibitem{Schiffler2014Quiver}
Ralf Schiffler.
\newblock {\em Quiver representations}.
\newblock CMS Books in Mathematics/Ouvrages de Math\'{e}matiques de la SMC. Springer, Cham, 2014.

\bibitem{Seiberg_1995}
N.~Seiberg.
\newblock Electric-magnetic duality in supersymmetric non-abelian gauge theories.
\newblock {\em Nuclear Physics B}, 435(1-2):129–146, Feb 1995.

\bibitem{Seong:2023njx}
Rak-Kyeong Seong.
\newblock {Unsupervised machine learning techniques for exploring tropical coamoeba, brane tilings and Seiberg duality}.
\newblock {\em Phys. Rev. D}, 108(10):106009, 2023.

\bibitem{sage}
W.\thinspace{}A. Stein et~al.
\newblock {\em {S}age {M}athematics {S}oftware ({V}ersion 9.4.0)}.
\newblock The Sage Development Team, 2021.
\newblock {\tt http://www.sagemath.org}.

\bibitem{Warkentin2014Exchange}
Matthias Warkentin.
\newblock {\em {Exchange Graphs via Quiver Mutation}}.
\newblock Doctoral dissertation, {Chemnitz University of Technology, Chemnitz}, 2014.

\end{thebibliography}

\end{document}